\documentclass[a4paper,10pt]{article}
\mathcode`\<="4268     % < = \langle
\mathcode`\>="5269     % > = \rangle
\mathcode`\:="603A     % : = \colon
\mathchardef\gt="313E  % arithmetic
\mathchardef\lt="313C  % strict order
\def\AA#1#2{\ensuremath{\xymatrix{#1\ar[r]^{\catebf{A}_{#1#2}}&#2}}}
\def\EE#1#2{\ensuremath{\xymatrix{#1\ar@{}[rr]^{\catebf{E}_{#1#2}}\ar[r]&\b&#2\ar[l]}}}
\def\II#1#2{\ensuremath{\xymatrix{#1\ar@{}[rr]^{\catebf{I}_{#1#2}}&\b\ar[l]\ar[r]&#2}}}
\def\OO#1#2{\ensuremath{\xymatrix{#1\ar@{}[rrr]^{\catebf{O}_{#1#2}}&\b\ar[l]\ar[r]&\b&#2\ar[l]}}}
\def\AAo#1#2{\ensuremath{\xymatrix{#2&#1\ar[l]_{(\catebf{A}_{#1#2})^{\circ}}}}}
\def\EEo#1#2{\ensuremath{\xymatrix{#2\ar@{}[rr]^{(\catebf{E}_{#1#2})^{\circ}}
\ar[r]&\b&#1\ar[l]}}}
\def\IIo#1#2{\ensuremath{\xymatrix{#2\ar@{}[rr]^{(\catebf{I}_{#1#2})^{\circ}}&\b
\ar[l]\ar[r]&#1}}}
\def\OOo#1#2{\ensuremath{\xymatrix{#2\ar@{}[rrr]^{(\catebf{O}_{#1#2})^{\circ}}\ar[r]&\b&
\b\ar[l]\ar[r]&#1}}}
\def\aa#1#2{\ensuremath{\catebf{A}_{#1#2}}}
\def\ee#1#2{\ensuremath{\catebf{E}_{#1#2}}}
\def\ii#1#2{\ensuremath{\catebf{I}_{#1#2}}}
\def\oo#1#2{\ensuremath{\catebf{O}_{#1#2}}}
\usepackage{amsthm}
\theoremstyle{definition}
\newtheorem{theorem}{Theorem}[section]
\newtheorem{lemma}[theorem]{Lemma}

\newtheorem{corollary}[theorem]{Corollary}
\newtheorem{definition}[theorem]{Definition}

\newtheorem{example}[theorem]{Example}
\usepackage[english]{babel}
\usepackage{rotating}
\usepackage[all]{xy}
\usepackage{txfonts}
\usepackage{dsfont}
\usepackage{bbm}
\usepackage{latexsym}
\title{A diagrammatic calculus of $n$-term syllogisms}
\author{Ruggero Pagnan\\
DISI, University of Genova, Italy\\
\texttt{ruggero.pagnan@disi.unige.it}}
\date{}
\begin{document}
\maketitle
\begin{abstract}
We extend the diagrammatic calculus of syllogisms introduced
in~\cite{Pagnan1} to the general case of $n$-term syllogisms, showing
that the valid ones are exactly those whose conclusion
follows by calculation. Moreover, by pointing out the existing
connections with the theory of rewriting systems we will also single out a suitable
category theoretic framework fo the calculus.\\
\linebreak
inference, rewriting system, rewrite rule, polygraph.
\end{abstract}
%
% QUESTO FILE CONTIENE LE DEFINIZIONI UTILIZZATE ALL'INTERNO DEL CODICE.
%DI SEGUITO SONO RIPORTATE IN ORIDNE ALFABETICONy
\def\a{\ensuremath{(a)}}
%-----------------------------------------------------------------------------------------
%#1 aggiunto sinistro a #2
\def\adj#1#2{\ensuremath{\xymatrix@C.2cm{#1\ar@{-|}[r]&#2}}}
%
%-----------------------------------------------------------------------------------------
%coppia di funtori aggiunti
\def\adjpair#1#2#3#4#5{\ensuremath{\xymatrix{#1\ar@/^/[r]^{#3}_{\hole}="1"&#2\ar@/^/[l]^{#4}_{\hole}="2" \ar@{} "2";"1"|(.3){#5}}}}
%l'aggiunto sinistro va da destra a sinistra ed Nh la freccia "sopra",
%il destro da sinistra a destra ed Nh la freccia "sotto".
% #1 Nh la categoria a sinistra, #2 Nh la categoria a destra, #3 il nome
% dell'aggiunto sinistro e #4 il nome dell'aggiunto destro, #3 che #4 sono del tipo
%^F e ^G rispettivamente.
% #5 Nh il simbolo di aggiunzione \perp
%
%congiunzione logica
\def\and{\ensuremath{\wedge}}
%
%freccia da sinistra a destra
\def\arr{\ensuremath{\rightarrow}}
%
%freccia doppia da sinistra a destra
\def\Arr{\ensuremath{\Rightarrow}}
%
%freccia nominata al centro
%\def\arrow#1#2#3{\ensuremath{\xymatrix@!{#1\ar[r]#3&#2}}}
%%#1 dominio, #2 codominio, #3 nome.
%N.B. #3 come ^f, _g, ^{fg}...
%
%categoria delle frecce
%
\def\arrow#1{\ensuremath{\cateb{#1}^{\arr}}}
%
%asterisco
\def\as{\ensuremath{\ast}}
%
%assioma della scelta
%\def\AS{\ensuremath{(AS)}}
%
%Bullet
\def\b{\ensuremath{\bullet}}
%
%condizione di Beck-Chevalley
\def\BC{\ensuremath{(BC)}}
%
%condizione di Beck-Chevalley per prodotti fibrati deboli
\def\BCd{\ensuremath{(BC)_d}}
\def\beps{\ensuremath{\backepsilon}}
%
%Bicategoria
\def\bicat#1{\ensuremath{\mathcal{#1}}}
%
%grassetto
\def\bf#1{\ensuremath{\mathbf{#1}}}
%
%frecce parallele lunghe
\def\bpararrow#1#2#3#4{\ensuremath{\xymatrix{#1\ar@<.9ex>[rr]^{#3}\ar@<-.9ex>[rr]_{#4}&&#2}}}
%
%diagramma quadrato "grande"
\def\bsquare#1#2#3#4#5#6#7#8{\xymatrix{#1\ar[dd]_{#8}\ar[rr]^#5&&#2\ar[dd]^{#6}\\
\\
#3\ar[rr]_{#7}&&#4}}
%#1 e #2, #3 e #4 sono gli oggetti del quadrato numerati in sequenza, prima riga
% e seconda riga da sinistra a destra.
%#5, #6, #7, #8 sono i nomi delle frecce numerati in senso orario a partire da quella
%che costituisce il la to superiore del quadrato.
%----------------------------------------------------
%bullet
\def\bu{\ensuremath{\bullet}}
\def\btoind#1{#1\index{#1@\textbf{#1}}}
%
%cardinalitN` di un insieme
\def\card#1{\ensuremath{|#1|}}
%------------------------------------------------------------------------------------------
%nome in \mathcal di una categoria
\def\cate#1{\ensuremath{\mathcal{#1}}}
%-----------------------------------------------------------------------------------------
%nome in blackboard bold di una categoria
\def\cateb#1{\ensuremath{\mathbbm{#1}}}
%
%nome in bold face
\def\catebf#1{\ensuremath{\mathbf{#1}}}
\def\catepz#1{\ensuremath{\mathpzc{#1}}}
%
%nome in \mathcal della categoria opposta
\def\cateop#1{\ensuremath{\cate#1^{op}}}
%-----------------------------------------------------------------------------------------
%freccia composta di due consecutive
\def\comparrow#1#2#3#4#5#6{\ensuremath{\xymatrix{#6#1\ar[r]#4&#2\ar[r]#5&#3}}}
%N.B. l'argomento #6 Nh della forma "gf:"
%-----------------------------------------------------------------------------------------
%B conseguenza logica di A
\def\cons#1{\ensuremath{\newdir{|>}{%
!/4.5pt/@{|}}\xymatrix@1@C=.3cm{\ar@{|=}[r]_{#1}&}}}
%-----------------------------------------------------------------------------------------
%la proprietN` (D)
\def\D{\ensuremath{(D)}}
%-----------------------------------------------------------------------------------------
%imposta e numera sequenzialmente gli enunciati
\def\enunciato#1#2#3#4#5#6#7{\newtheorem{#4}#5{#1}#6                                              %IL TIPO DI ENUNCIATO
\begin{#4}#2\label{#7}
#3
\end{#4}}
% #1  Nh lemma, proposizione, definizione, teorema...
% #2  Nh [commento all'enunciato] come per esempio [teorema dei due carabinieri]
% #3  Nh l'enunciato vero e proprio
% #4  Nh il nome che si assegna a piacere all'enunciato
% #5  Nh il nome che si Nh assegnato al primo enunciato della sezione in cui si sta %       %     lavorando racchiuso tra parentesi quadre
% #6  Nh [section], ma Nh da mettere solo al primo enunciato della sezione
% #7  Nh l'etichetta alla quale eventualmente si vorrN` fare riferimento in seguito per citare
%     l'enunciato
%-----------------------------------------------------------------------------------------
%
\def\dfn#1#2#3#4#5#6#7{\newtheorem{#4}#5{#1}#6                                              %IL TIPO DI ENUNCIATO
\begin{#4}#2\label{#7}
\emph{#3}
\end{#4}}
% #1  Nh lemma, proposizione, definizione, teorema...
% #2  Nh [commento all'enunciato] come per esempio [teorema dei due carabinieri]
% #3  Nh l'enunciato vero e proprio
% #4  Nh il nome che si assegna a piacere all'enunciato
% #5  Nh il nome che si Nh assegnato al primo enunciato della sezione in cui si sta %       %     lavorando racchiuso tra parentesi quadre
% #6  Nh [section], ma Nh da mettere solo al primo enunciato della sezione
% #7  Nh l'etichetta alla quale eventualmente si vorrN` fare riferimento in seguito per citare
%     l'enunciato
%
%freccia epi estremale
\def\epiarrow#1#2#3#4{\ensuremath{\newdir{|>}{%
!/4.5pt/@{|}*:(1,-.2)@^{>}*:(1,+.2)@_{>}}\xymatrix{#3#1\ar@{-|>}[r]#4&#2}}}
%-----------------------------------------------------------------------------------------
%punta di epi estremale
\def\epitip{\ensuremath{\newdir{|>}{%
!/4.5pt/@{|}*:(1,-.2)@^{>}*:(1,+.2)@_{>}}}}
%-----------------------------------------------------------------------------------------
%lettera epsilon
\def\eps#1{\ensuremath{\epsilon#1}}
\def\esempio#1#2#3#4#5#6#7{\newtheorem{#4}#5{#1}#6                                              
%IL TIPO DI ENUNCIATO
\begin{#4}#2\label{#7}
\textup{#3}
\begin{flushright}
$\blacksquare$
\end{flushright}
\end{#4}}
% #1  Nh lemma, proposizione, definizione, teorema...
% #2  Nh [commento all'enunciato] come per esempio [teorema dei due carabinieri]
% #3  Nh l'enunciato vero e proprio
% #4  Nh il nome che si assegna a piacere all'enunciato
% #5  Nh il nome che si Nh assegnato al primo enunciato della sezione in cui si sta %       %     lavorando racchiuso tra parentesi quadre
% #6  Nh [section], ma Nh da mettere solo al primo enunciato della sezione
% #7  Nh l'etichetta alla quale eventualmente si vorrN` fare riferimento in seguito per citare
%     l'enunciato
%
%lettera eta
\def\et#1{\ensuremath{\eta#1}}
%-----------------------------------------------------------------------------------------
%Funtore "immagine diretta"
\def\Ex#1{\mbox{\boldmath\ensuremath{\exists}}{#1}}
%-----------------------------------------------------------------------------------------
%esistenziale
\def\ex#1{\ensuremath{\exists#1}}
%-----------------------------------------------------------------------------------------
%condizione piNy debole di (R)
\def\F{\ensuremath{(F)}}
%-----------------------------------------------------------------------------------------
%freccia nominata in fronte
\def\farrow#1#2#3{\ensuremath{\xymatrix{#3:#1\ar[r]&#2}}}
%#1 dominio, #2 codominio, #3 nome
%-----------------------------------------------------------------------------------------
%lettera Fi maiuscola
\def\Fi{\ensuremath{\Phi}}
%
%fibrazione
\def\fib#1#2#3{\ensuremath{
\catebf{#3}:\cateb{#1}\arr\cateb{#2}}}
%
%fibrazione con categoria totale slice
%
\def\fibs#1#2#3#4{\ensuremath{\begin{array}{ll}
\slice{\catebf{#1}}{#2}\\
\,\downarrow^{\catebf{#4}}\\
\catebf{#3}
\end{array}}}
%fibrazione con categoria totale e di base slice
\def\fibss#1#2#3#4#5{\ensuremath{\begin{array}{ll}
\slice{\catebf{#1}}{#2}\\
\,\downarrow^{#5}\\
\slice{\catebf{#3}}{#4}
\end{array}}}
\def\forev#1{\ensuremath{\forall~#1,~\forall~}}
%-----------------------
%categoria di frazioni destre
\def\fract#1{\ensuremath{\cateb{#1}[\Sigma^{-1}]}}
%
%Grothendieck construction
\def\G#1{\ensuremath{\int{#1}}}
%
%Hom set di una categoria
%\def\Hom#1{\ensuremath{\mathit{Hom}#1}}
%
%freccia arpionata in alto
\def\harrow{\ensuremath{\rightharpoonup}}
%
%freccia uncinata
\def\hook{\ensuremath{\hookrightarrow}}
%-----------------------------------------------------------------------------------------
%Hom set di una categoria
\def\hset#1#2#3{\ensuremath{\cate{#1}(#2,#3)}} %Hom set
%-----------------------------------------------------------------------------------------
%immagine
\def\Im#1{\ensuremath{\mathit{Im}#1}}
%-----------------------------------------------------------------------------------------
%implicazione logica
\def\implies{\ensuremath{\Rightarrow}}
%-----------------------------------------------------------------------------------------
%prodotto indiciato
\def\indprod#1#2{\ensuremath{\prod#1#2}}
%
%doppia linea di inferenza
\def\infer#1#2{\ensuremath{\left.\begin{array}{cc}
#1\\
\xymatrix{\ar@{-}@<-.3ex>[rrrr]\ar@{-}@<.3ex>[rrrr]&&&&}\\
#2\\
\end{array}\right.}}
%
%predicato di appartenenza "relativizzato"
%\def\inn#1{\ensuremath{\in_{#1}}}
%
%italico in ambiante mate
\def\it#1{\ensuremath{\mathit{#1}}}
%
%the lift object
\def\l#1{\ensuremath{#1_{\bot}}}
\def\larrow{\ensuremath{\leftarrow}}
%-----------------------------------------------------------------------------------------
%composta "lunga" di due frecce
\def\lcomparrow#1#2#3#4#5{\ensuremath{\xymatrix{#1\ar[rr]#4&&#2\ar[rr]#5&&#3}}}
%-----------------------------------------------------------------------------------
%categoria di frazioni sinistre
\def\lfract#1{\ensuremath{[\Sigma^{-1}]\cateb{#1}}}
%
%freccia uncinata lunga per sotto_oggetti
\def\lhookarrow#1#2#3{\ensuremath{\xymatrix{#1\ar@{^{(}->}[rr]^{#3}&&#2}}}
%-----------------------------------------------------------------------------------------
%categoria dei modelli della teoria
\def\mod#1#2{\ensuremath{\cateb{Mod}(#1,#2)}}
%-----------------------------------------------------------------------------------------
%la collezione dei mono di codominio un oggetto fissato
\def\Mono#1{\ensuremath{\mathit{Mono}#1}}
%-----------------------------------------------------------------------------------------
%freccia mono
\def\monoarrow#1#2#3#4{\ensuremath{\newdir{
>}{{}*!/-5pt/@{>}}\xymatrix{#3#1\ar@{ >->}[r]#4&#2}}}
%-----------------------------------------------------------------------------------------
%coda di un mono
\def\monotail{\ensuremath{\newdir{
>}{{}*!/-5pt/@{>}}}}
%-----------------------------------------------------------------------------------------
%assioma (N)
\def\N{\ensuremath{(N)}}
%trasformazione naturale tra due funtori
\def\natarrow#1#2#3{\ensuremath{\newdir{|>}{%
!/4.5pt/@{|}*:(1,-.2)@^{>}*:(1,+.2)@_{>}}\xymatrix{#3:#1\ar@{=|>}[r]&#2}}}
%-----------------------------------------------------------------------------------------
%notazione per freccia che rappresenta una trasformazione naturale
\def\natbody{\ensuremath{\newdir{|>}{%
!/4.5pt/@{|}*:(1,-.2)@^{>}*:(1,+.2)@_{>}}}}
%-----------------------------------------------------------------------------------------
%due funtori paralleli e una trasformazione naturale tra di essi
\makeatletter
\def\natpararrow#1#2#3#4#5{\@ifnextchar(
 {\Natpararrow{#1}{#2}{#3}{#4}{#5}}
 {\Natpararrow{#1}{#2}{#3}{#4}{#5}(.5)}}
\makeatother
\def\Natpararrow#1#2#3#4#5(#6){\ensuremath{\newdir{|>}{%
!/4.5pt/@{|}*:(1,-.2)@^{>}*:(1,+.2)@_{>}}
\xymatrix{#1\ar@<1.5ex>[rr]^(#6){#3}|(#6){\vrule height0pt depth.7ex
width0pt}="a"\ar@<-1.5ex>[rr]_(#6){#4}|(#6){\vrule height1.5ex width0pt}="b"&&#2
\ar@{=|>} "a";"b"#5}}}
%-----------------------------------------------------------------------------------------
%lettera omega minuscola
\def\om{\ensuremath{\omega}}
%
%sopralineato
\def\ov#1{\ensuremath{\overline{#1}}}
%
%sequenza ordinata di componenti
\def\ovarr#1{\ensuremath{\overrightarrow{#1}}}
%
%frecce parallele
\def\pararrow#1#2#3#4{\ensuremath{\xymatrix{#1\ar@<.8ex>[r]^{#3}\ar@<-.8ex>[r]_{#4}&#2}}}
%-----------------------------------------------------------------------------------------
%la collezione dei sotto-oggetti piccoli di un oggetto
\def\p#1{\ensuremath{\mathds{P}_{\catepz{S}}}(#1)}
\def\P#1{\ensuremath{\mathnormal{P}(#1)}}
%-----------------------------------------------------------------
%proof
\def\pr#1{\textbf{Proof:~}#1 
%\begin{proof}[\textup{\textbf{Proof:~}}]#1\end{proof}
\begin{flushright}
$\Box$
\end{flushright}
}
%
%\def\qed{{%       set up
 % \parfillskip=0pt       % so \par doesnt push \square to left                                    %\widowpenalty=10000     % so we dont break the page before \square
 %\displaywidowpenalty=10000 % ditto
 %\finalhyphendemerits=0 % TeXbook exercise 14.32
%
%                 horizontal
% \leavevmode             % \nobreak means lines not pages
 %\unskip                 % remove previous space or glue
 %\nobreak                % don$(B!G(Bt break lines
 %\hfil                   % ragged right if we spill over
 %\penalty50              % discouragement to do so
 %\hskip.2em              % ensure some space
 %\null                   % anchor following \hfill
 %\hfill                  % push \square to right
 %$\square$%              % the end-of-proof mark
%
%                   vertical
 %\par}}                  % build paragraph
%
%freccia "barrata" per relazione tra due oggetti.
\def\rel#1#2#3{\ensuremath{\xymatrix{#1:#2\ar[r]
|-{\SelectTips{cm}{}\object@{|}} &#3}}}
%
%\def\rell{\ensuremath{{\SelectTips{cm}{}\object@{|}}}} %la barra
                %nella freccia per relazioni.
%
%restrizione
\def\res#1#2{\ensuremath{#1_{\rbag #2}}}
\def\S#1{\ensuremath{\catepz{S}_{\cateb{#1}}}}
\def\sectoind#1#2{#1\index{#2!#1}}
%
%linea di inferenza semplice
\def\sinfer#1#2{\ensuremath{\left.\begin{array}{cc}
#1\\
\xymatrix{\ar@{-}[rrrr]&&&&}\\
#2\\
\end{array}\right.}}
%
%%categoria slice
\def\slice#1#2{\ensuremath{#1/{\textstyle#2}}}
\def\sslice#1#2{\ensuremath{#1//{\textstyle#2}}}
%-----------------------------------------------------------------------------------------
%diagramma quadrato "piccolo" con oggetti e "frecce normali". Gli oggetti corrispondono
%ai parametri #1 e #2, #3 e #4, dall'alto al basso prima riga e seconda riga, da sinistra
% a destra. Le frecce corrispondono ai parametri #5, #6, #7, #8 fornite in senso orario
%a partire dal nome della freccia che costituisce il lato superiore del quadrato.
\def\square#1#2#3#4#5#6#7#8{$$\xymatrix{#1\ar[d]_{#8}\ar[r]^{#5}&#2\ar[d]^{#6}\\
#3\ar[r]_{#7}&#4}$$}
%N.B. il quadrato Nh giN` centrato
%-----------------------------------------------------------------------------------------
%notazione per sotto-oggetto
\def\sub#1#2#3{\ensuremath{#1(#2#3)~}}
%-----------------------------------------------------------------------------------------
%funtore sigma maiuscolo
\def\summ#1{\ensuremath{\sum#1}}
%
%tilde
\def\ti#1{\ensuremath{\tilde#1}}
%
%prodotto tensore
\def\tens{\ensuremath{\otimes}}
%
%la teoria \tau
\def\teo#1{\ensuremath{\tau#1}}
%-----------------------------------------------------------------------------------------
\def\ter#1{\ensuremath{\tau_{\cate{#1}}}}
%-----------------------------------------------------------------------------------------
%lettera teta minuscola
\def\teta{\ensuremath{\theta}}
%
%voce all'indice analitico
%
\def\toind#1{#1\index{#1}}
%
%sottolineato
\def\un#1{\underline{#1}}
%--------------------------------------------------------
%2-cella verticale tra due frecce parallele
\def\vcell{\ensuremath{\ar@<-1ex>@{}[r]_{\hole}="a"\ar@<+1ex>@{}[r]^{\hole}="b"}}
%
%lettera epsilon
\def\veps{\ensuremath{\varepsilon}}
%
%lettera fi minuscola
\def\vfi{\ensuremath{\varphi}}
\def\w{\ensuremath{\wedge}}
\def\wti#1{\ensuremath{\widetilde#1}}
\def\y#1{\ensuremath{y#1}}
%-----------------------------------------------------------------------------------------
%immersione di Yoneda
%\def\yon#1{\ensuremath{\cateb{Y}#1}}
%-----------------------------------------------------------------------------------------

%
\section{Introduction}
The main aims of the present paper are on one hand that of extending to 
the $n$-term case the diagrammatic calculus of syllogisms 
introduced in~\cite{Pagnan1}, where we
dealt with the basic $3$-term case and on the other hand that of 
single out a suitable category theoretic framework for the calculus itself.
To the author's knowledge, another diagrammatic approach based
on directed graphs already exists, see~\cite{Smyth}, whereas for a
category theoretic point of view, the reader may
consult~\cite{MR1271697}.\\
In section~\ref{prelsyll} we briefly recall the basics on syllogisms
and the the diagrammatic calculus we hinted at above. 
In section~\ref{nsyll}, we will deal with $n$-term
syllogisms and prove that the calculus extends to them, by showing in
turn that the valid $n$-term syllogisms are exactly those whose
conclusion follows from their 
premisses by calculation. Moreover, we will also retrieve the well-known result that
the valid $n$-term syllogisms are $3n^2-n$. In section~\ref{just}, we
will point out the existing connections with the theory of rewriting
systems, by approaching them through polygraphs, mainly referring
to~\cite{MR1224519}.
\section{Preliminaries on syllogisms}\label{prelsyll}
We will refer to nouns, adjectives or more
complicated expressions of the natural language as to terms,
generically, and denote them by using upper case letters which we
will also call \emph{term-variables}.\\
The first systematization of syllogistic is due to Aristotle. Since him
the following four kinds of propositions were
recognized as fundamental throughout the research in logic:
\begin{eqnarray*}
&&\mbox{$\catebf{A}_{AB}$: All A is B (universal affirmative proposition)}\\
&&\mbox{$\catebf{E}_{AB}$: No A is B (universal negative proposition)}\\
&&\mbox{$\catebf{I}_{AB}$: Some A is B (particular affirmative proposition)}\\
&&\mbox{$\catebf{O}_{AB}$: Some A is not B (particular negative proposition)}
\end{eqnarray*}
Following the tradition that is, loosely, the medieval systematization of syllogistic,
we will henceforth refer them to as 
\emph{categorical propositions}. In each of them, $A$ denotes the
\emph{subject} whereas $B$ denotes the \emph{predicate} of the
corresponding proposition.\\
A \emph{syllogism} is a rule of inference that involves three categorical
propositions that are distinguished by referring them to as
\emph{first premise}, \emph{second premise} and \emph{conclusion}.
Moreover, a syllogism involves exactly three term-variables $S$, $P$ and
$M$ in the following precise way: $M$ does not
occur in the conclusion whereas, according to the tradition, $P$ occurs in the first premise and 
$S$ occurs in the second premise. The term-variables $S$ and $P$ 
occur as the subject and predicate of the conclusion, respectively,
and are also referred to as \emph{minor term} and \emph{major term} of
the syllogism, whereas $M$ is also referred to as \emph{middle
term}.\\
The \emph{mood} of a syllogism is the sequence of the kinds of 
categorical propositions by which it is formed, whereas its
\emph{figure} is the position of the term-variables $S$, $P$ and $M$ in it.
There are four possible figures, as shown in the table
\begin{eqnarray}\label{figs}
\begin{tabular}{|l|c|c|c|c|}
\hline
& fig. 1 & fig. 2 & fig. 3 & fig. 4\\
\hline
first premise & MP & PM & MP & PM\\
\hline
second premise & SM & SM & MS & MS\\
\hline
conclusion & SP & SP & SP & SP\\
\hline
\end{tabular}
\end{eqnarray}
and a syllogism is completely determined by its mood and by its figure
together. We write syllogisms so that their mood and figure can be
promptly retrieved and let the symbol $\models$
separate the premisses from the conclusion.
For example, in the syllogism
\begin{eqnarray}\label{sylloprimo}
\catebf{A}_{MP},\catebf{A}_{SM}\models\catebf{A}_{S P}
\end{eqnarray}
it is possible to recognize from left to right the first premise, the
second premise and the conclusion, moreover the mood, which is
\catebf{AAA}, and the figure which is the first one. The combination
of the moods and figures gives rise to 256 syllogisms in total, of
which only 24 are \emph{valid}, that is such that the
verification of the premisses necessarily entails the verification of
the conclusion. Venn diagrams
can be used to verify the validity of syllogisms, see~\cite{MR0037255}
for example. We hasten to say that 
of the 24 valid syllogisms, 9 are valid under suitable additional
assumptions and will be henceforth referred to
as \emph{syllogisms with assumption of existence} for reasons that
will be cleared later on, whereas the remaining 15 are valid 
without any further assumption and in the present section we continue refer them to as
syllogisms, simply. These are the ones listed in the table
\begin{eqnarray}\label{valsyll}
\begin{tabular}{|c|c|c|c|}
\hline
fig. 1 & fig. 2 & fig. 3 & fig. 4\\
\hline
\catebf{AAA}&\catebf{EAE}&\catebf{IAI}&\catebf{AEE}\\
\catebf{EAE}&\catebf{AEE}&\catebf{AII}&\catebf{IAI}\\
\catebf{AII}&\catebf{EIO}&\catebf{OAO}&\catebf{EIO}\\
\catebf{EIO}&\catebf{AOO}&\catebf{EIO}&\\
\hline
\end{tabular}
\end{eqnarray}

Now, for the previously cited diagrammatic calculus of syllogisms,
graphical representations of the categorical
propositions are correspondingly given, that is
\[\AA{A}{B}\qquad\EE{A}{B}\]
\[\II{A}{B}\qquad\OO{A}{B}\]
which will be henceforth referred to as
\emph{Aristotelian diagrams}. Each of them
has a corresponding dual, namely
\[\AAo{B}{A}\qquad\EEo{B}{A}\]
\[\IIo{B}{A}\qquad\OOo{B}{A}.\]
Two or more Aristotelian diagrams, and their duals, can be
concatenated and reduced, if possible. In such a concatenation, a
reduction applies by formally composing two or more consecutive
and accordingly oriented arrow symbols separated by a single
term-variable, thus deleting it. Such a reduction will be
henceforth referred to as \emph{syllogistic inference}. By means
of syllogistic inferences, Aristotelian diagrams can be used to
verify the validity of syllogisms. This is obtained by using three
Aristotelian diagrams, as the first premise, the second premise,
and the conclusion of the syllogism. Moreover, these involve three
distinguished term-variables, denoted $S$, $P$ and $M$, in such a
way that $M$ occurs in both the Aristotelian diagrams in the
premisses and does not in the conclusion, whereas $S$ and $P$
occur in the conclusion as well as in the premisses. Following
the tradition, $P$ will occur in the first premise whereas $S$ in the
second. Syllogistic inferences will be represented by diagrams filled in
with the symbol $\models$ upside down, so to explicitly underline
the fact that the notion of syllogistic inference is a directed
one but also written in line, by letting $\sharp$ denote the
concatenation of Aristotelian diagrams. 
Thus, for example, the syllogistic inference associated with
the valid syllogisms~(\ref{sylloprimo}) can be either diagrammatically
represented as
\[\xymatrix@R=1ex{S\ar@{=}[d]\ar@{}[drr]|{\begin{turn}{270}$\models$\end{turn}}
\ar[r]^{\aa{S}{M}}&M\ar[r]^{\aa{M}{P}}&P\ar@{=}[d]\\
S\ar[rr]_{\aa{S}{P}}&&P}\]
or written as
\[(\aa{M}{P})\sharp(\aa{S}{M})\models(\aa{S}{P}).\]
Validity of syllogism 
\begin{eqnarray}\label{syllat}
\catebf{A}_{PM},\catebf{E}_{MS}\models\catebf{E}_{SP}
\end{eqnarray}
is witnessed by the existence of a syllogistic inference
reducing the concatenation of the Aristotelian diagrams for
its premisses to the Aristotelian diagram for its
conclusion. The concatenation of the Aristotelian
diagrams for the premisses of~(\ref{syllat}) is 
\[(\catebf{A}_{PM})^{\circ}\sharp(\catebf{E}_{MS})^{\circ}\]
that is
\[\xymatrix{S\ar@{}[rr]^{(\ee{M}{S})^{\circ}}\ar[r]&\b&M\ar[l]&P\ar[l]_{(\aa{P}{M})^{\circ}}}\]
whereas in its entirety
the syllogistic inference can be written as 
\begin{eqnarray}\label{syllinf}
(\catebf{A}_{PM})^{\circ}\sharp(\catebf{E}_{MS})^{\circ}\models(\catebf{E}_{SP})
\end{eqnarray}
or represented by the diagram
\begin{eqnarray}\label{firstex}
\xymatrix@R=1ex{S\ar@{}[rr]^{(\ee{M}{S})^{\circ}}
\ar@{}[drrr]|{\begin{turn}{270}$\models$\end{turn}}
\ar@{=}[d]\ar[r]&\b&M\ar[l]&P\ar[l]_{(\aa{P}{M})^{\circ}}\ar@{=}[d]\\
S\ar@{}[rrr]_{\ee{S}{P}}\ar[r]&\b&&P\ar[ll]}
\end{eqnarray}
to produce evidence for~(\ref{syllinf}), since in~(\ref{firstex}) 
the Aristotelian diagram representing the conclusion of~(\ref{syllat})
has been obtained by reduction through the formal calculation of the composite 
$\b\leftarrow M\leftarrow P$, making the middle term $M$ disappear.
We hasten to remark that the unlabelled version of 
diagram~(\ref{firstex}), namely diagram
\begin{eqnarray}\label{firstexx}
\xymatrix@R=1ex{S\ar@{}[drrr]|{\begin{turn}{270}$\models$\end{turn}}
\ar@{=}[d]\ar[r]&\b&M\ar[l]&P\ar[l]\ar@{=}[d]\\
S\ar[r]&\b&&P\ar[ll]}
\end{eqnarray}
does not uniquely determine the
syllogistic inference~(\ref{syllinf}), since it must be taken into
account that the same mood can occur in more than one figure, as
clearly shown by table~(\ref{valsyll}). In general, 
the unlabelled diagram of a syllogistic inference 
determines the mood of a syllogism only up to figure. For instance, 
diagram~(\ref{firstexx}) produces evidence for 
the syllogistic inference
\[(\catebf{A}_{PM})^{\circ}\sharp(\catebf{E}_{SM})\models(\catebf{E}_{SP})\]
by relabelling it as
\[\xymatrix@R=1ex{S\ar@{}[rr]^{\ee{S}{M}}\ar@{}[drrr]|{\begin{turn}{270}$\models$\end{turn}}
\ar@{=}[d]\ar[r]&\b&M\ar[l]&P\ar[l]_{(\aa{P}{M})^{\circ}}\ar@{=}[d]\\
S\ar@{}[rrr]_{\ee{S}{P}}\ar[r]&\b&&P\ar[ll]}\]
thus validating the mood $\catebf{AEE}$ in the second figure,
namely the syllogism
\[\catebf{A}_{PM},\catebf{E}_{SM}\models\catebf{E}_{SP}\]
whereas the syllogistic inference~(\ref{syllinf}) was validating the
mood $\catebf{AEE}$ in the fourth figure. 
Now, we let the reader convince herself that suitable labellings of  the diagram
\[\xymatrix@R=1ex{S\ar@{=}[d]\ar@{}[drrrr]|{\begin{turn}{270}$\models$\end{turn}}&\b\ar[l]\ar[r]&
M\ar[r]&\b&P\ar[l]\ar@{=}[d]\\
S&\b\ar[l]\ar[rr]&&\b&P\ar[l]}\]
produce evidence for the syllogistic inferences
\[(\catebf{E}_{MP})\sharp(\catebf{I}_{SM})\models(\catebf{O}_{SP})\]
\[(\catebf{E}_{PM})^{\circ}\sharp(\catebf{I}_{SM})\models(\catebf{O}_{SP})\]
\[(\catebf{E}_{MP})\sharp(\catebf{I}_{MS})^{\circ}\models(\catebf{O}_{SP})\]
\[(\catebf{E}_{PM})^{\circ}\sharp(\catebf{I}_{MS})^{\circ}\models(\catebf{O}_{SP})\]
validating the mood \catebf{EIO} in all the
figures.\\

A feature of the calculus at issue is that in a syllogistic inference, no bullet
symbol gets deleted. More precisely, for a valid syllogism,
the Aristotelian diagram for its conclusion contains as
many bullet symbols as in the Aristotelian diagrams for its
premisses. This fact turns out to be useful in showing that a
syllogism is not valid. For example, the syllogism
\begin{eqnarray}\label{notv}
\catebf{O}_{PM},\catebf{E}_{MS}\models\catebf{I}_{SP}
\end{eqnarray}
is not valid since if it were such, then the existence of the
syllogistic inference
\[(\catebf{O}_{PM})^{\circ}\sharp(\catebf{E}_{MS})^{\circ}\models(\catebf{I}_{SP})\]
would be witnessed by a diagram such as 
\begin{eqnarray}\label{previous}
\xymatrix@R=1ex{S\ar@{}[drrrrr]|{\begin{turn}{270}$\models$\end{turn}}
\ar@{}[rr]^{(\ee{M}{S})^{\circ}}\ar@{=}[d]\ar[r]&\b&
M\ar@{}[rrr]^{(\oo{P}{M})^{\circ}}\ar[l]\ar[r]&\b&\b\ar[l]\ar[r]&P\ar@{=}[d]\\
S\ar@{}[rrrr]_{\ii{S}{P}}&\b\ar[l]\ar[rrrr]&&&&P}
\end{eqnarray}
which fact is impossible since a single bullet symbol occurs in the
Aristotelian diagram for the 
conclusion, whereas three of them occur in those
for the premisses. This fact could be observed
even without drawing the previous diagram but by simply looking
at~(\ref{notv}). However, this criterion not always apply. It suffices to
consider the syllogistic inference
\[\xymatrix@R=1ex{S\ar@{=}[d]\ar@{}[drrrr]|{\begin{turn}{270}$\models$\end{turn}}
\ar[r]&\b&M\ar[l]&\b\ar[r]\ar[l]&P\ar@{=}[d]\\
S\ar[r]&\b&&\b\ar[ll]\ar[r]&P}\]
in which as many bullet symbols occur in the premisses as in the
conclusion, that we could be tempted to label as
$(\oo{P}{S})^{\circ}$, but doing this
would mean the interchanging of the r\^oles played by the
term-variables $S$ and $P$. On the other hand,
syllogism~(\ref{notv}) is not valid even because in
diagram~(\ref{previous}) $M$ is not erasable.\\

For every term-variable $A$, particularly interesting
instances of Aristotelian diagrams are the following:
\[\AA{A}{A}\qquad\EE{A}{A}\]
\[\II{A}{A}\qquad\OO{A}{A}\]
which have to be read as
\begin{eqnarray*}
&&\mbox{$\catebf{A}_{AA}$: All A is A}\\
&&\mbox{$\catebf{E}_{AA}$: No A is A}\\
&&\mbox{$\catebf{I}_{AA}$: Some A is A}\\
&&\mbox{$\catebf{O}_{AA}$: Some A is not A}
\end{eqnarray*}
respectively. The diagrams $\catebf{A}_{AA}$ and $\catebf{I}_{AA}$ are
referred to as \emph{laws of identity}, see~\cite{Lukasiewicz}.
In particular, $\catebf{I}_{AA}$ will be referred to as an
\emph{assumption of existence}, since it
affirms the inhabitation of $A$ whereas, on the contrary, $\catebf{E}_{AA}$
affirms its emptyness. The diagram $\catebf{O}_{AA}$ is an expression of
the \emph{principle of contradiction}, which fact
has been discussed in~\cite{Pagnan1}, to which we refer the interested
reader.\\

A \emph{syllogism with assumption of existence} is a syllogism 
that is valid under an additional assumption of existence of the form $\catebf{I}_{SS}$,
$\catebf{I}_{MM}$ or $\catebf{I}_{PP}$. The table 
\begin{eqnarray}\label{valhypsyll}
\begin{tabular}{|c|c|c|c|l|}
\hline
fig. 1 & fig. 2 & fig. 3 & fig. 4 & assumption\\
\hline
\catebf{AAI}&\catebf{AEO}&&\catebf{AEO}& $\catebf{I}_{SS}$\\
%
%\hline
%
\catebf{EAO}&\catebf{EAO}&&& $\catebf{I}_{SS}$\\
%
%\hline
%
&&\catebf{AAI}&\catebf{EAO}& $\catebf{I}_{MM}$\\
%
%\hline
%
&&\catebf{EAO}&& $\catebf{I}_{MM}$\\
%
%\hline
%
&&&\catebf{AAI}& $\catebf{I}_{PP}$\\
\hline
\end{tabular}
\end{eqnarray}
lists the valid syllogisms with assumption of existence. For instance,
in order to show that the syllogism with assumption of existence
\[\catebf{E}_{MP},\catebf{A}_{SM},\catebf{I}_{SS}\models\catebf{O}_{SP}\]
is valid, it suffices to consider the syllogistic inference 
\[(\catebf{E}_{MP})\sharp(\catebf{A}_{SM})\sharp(\catebf{I}_{SS})\models(\catebf{O}_{SP})\]
witnessed by diagram
\[\xymatrix@R=1ex{S\ar@{}[rr]^{\ii{S}{S}}\ar@{}[drrrrr]|{\begin{turn}{270}$\models$\end{turn}}
\ar@{=}[d]&\b\ar[r]\ar[l]&S\ar[r]^{\aa{S}{M}}&M\ar@{}[rr]^{\ee{M}{P}}
\ar[r]&\b&P\ar[l]\ar@{=}[d]\\
S\ar@{}[rrrrr]_{\oo{S}{P}}&\b\ar[l]\ar[rrr]&&&\b&P\ar[l]}\]
whose unlabelled version 
produces evidence for the syllogistic
inference
\[(\catebf{E}_{PM})^{\circ}\sharp(\catebf{A}_{SM})\sharp(\catebf{I}_{SS})\models(\catebf{O}_{SP})\]
too, validating the syllogism with assumption of existence
\[\catebf{E}_{PM},\catebf{A}_{SM},\catebf{I}_{SS}\models\catebf{O}_{SP}.\]

We end the section by citing
\begin{theorem}\label{main}
A syllogism (with assumption of existence) is valid if and only if
there is a necessarily unique syllogistic inference from its premisses to its conclusion.
\end{theorem}
\begin{proof}
See~\cite{Pagnan1}.
\end{proof}
\section{$n$-term syllogisms}\label{nsyll}
Whereas syllogisms, either with assumption of existence or not, 
involve exactly $3$ term-variables, $n$-term 
syllogisms involve exactly $n$ term-variables $A_1,\ldots,A_n$, $n\geq 1$, linked by
$n$ categorical propositions any two contiguous of which have
exactly one term in common. We may represent the $n$ categorical
propositions as 
\[\catebf{X}_{A_{n-1}A_n},\catebf{X}_{A_{n-2}A_{n-1}},\ldots,\catebf{X}_{A_2A_3},
\catebf{X}_{A_1A_2}\catebf{X}_{A_1A_n}\]
where \catebf{X} is a symbol
between \catebf{A}, \catebf{E}, \catebf{I}, \catebf{O} and, for every
$i=1,\ldots,n-1$, $\catebf{X}_{A_iA_{i+1}}$ stays either for
$\catebf{X}_{A_iA_{i+1}}$ or for $\catebf{X}_{A_{i+1}A_i}$. 
We write
\[\catebf{X}_{A_{n-1}A_n},\catebf{X}_{A_{n-2}A_{n-1}},\ldots,\catebf{X}_{A_2A_3},
\catebf{X}_{A_1A_2}\models\catebf{X}_{A_1A_n}\]
to denote a generic $n$-term syllogism. We remark that
possibly occurring assumptions of existence of the form
$\catebf{I}_{A_iA_i}$, for some $i=1,\ldots,n$, will be explicitly
mentioned when needed. In doing this, we let the 
expression ``$n$-term syllogism'' comprise the case
in which no assumption of existence occurs as well as the case in which such an
assumption occurs.\\
It is well known that the total number of valid $n$-term syllogisms is
$3n^2-n$, see~\cite{Meredith}, where such a formula was obtained 
by rejecting the not valid moods on the bases of the traditional rules of
syllogism. The same formula was reobtained in~\cite{Smyth} by a 
diagrammatic method allowing a direct calculation.\\
The aim of the present section is that of generalize 
theorem~\ref{main} to the case of $n$-term
syllogisms and simulataneosly that of directly recalculate the previously cited
formula by using syllogistic inferences. 
\begin{lemma}\label{nolemma}
For every positive natural number $n$, 
a syllogistic inference yields an Aristotelian diagram 
as a conclusion in exactly the following cases:
\begin{itemize}
\item[(i)] $A_1\arr A_2\arr\cdots\arr A_i\arr A_{i+1}\arr\cdots\arr
  A_{n-1}\arr A_n$.
\item[(ii)] $A_1\arr A_2\arr\cdots\arr A_i\arr\b\leftarrow A_{i+1}\leftarrow
  \cdots\leftarrow A_{n-1}\leftarrow A_n$, with $1\leq i \leq n-1$.
\item[(iii)] $A_1\leftarrow A_2\leftarrow\cdots\leftarrow A_i\leftarrow\b\arr
  A_{i+1}\arr\cdots\arr A_{n-1}\arr A_n$, with $1\leq i\leq n-1$.
\item[(iv)] $A_1\leftarrow A_2\leftarrow\cdots\leftarrow A_i\leftarrow\b\arr
  A_i\arr\cdots\arr A_{n-1}\arr A_n$, with $1\leq i\leq n$.
\item[(v)] $A_1\leftarrow A_2\leftarrow\cdots\leftarrow
  A_i\leftarrow\b\arr\b\leftarrow A_{i+1}\leftarrow\cdots\leftarrow
  A_{n-1}\leftarrow A_n$, with $1\leq i\leq n-1$. 
\item[(vi)] $A_1\leftarrow\cdots\leftarrow A_i\leftarrow\b\arr
A_{i+1}\arr\cdots\arr A_{j-1}\arr\b\leftarrow
A_j\leftarrow\cdots\leftarrow A_n$, with $1\leq i\lt j\leq n$.
\item[(vii)] $A_1\leftarrow\cdots\leftarrow A_i\leftarrow\b\arr
A_i\arr\cdots\arr A_{j-1}\arr\b\leftarrow
A_j\leftarrow\cdots\leftarrow A_n$, with $1\leq i\lt j\leq n$.
\end{itemize}
\end{lemma}
\begin{proof}
It is clear that a syllogistic inference applies to each of the
diagrams listed in the statement yielding an Aristotelian diagram 
involving the terms $A_1$ and $A_n$ only. Conversely, we proceed by cases:
\begin{itemize}
\item[(a)] the only way to obtain $A_1\arr A_n$ as a conclusion of a
  syllogistic inference is by (i), since no bullet symbol is allowed to occur
  in the conclusion.
\item[(b)] the only way to obtain $A_1\arr\b\leftarrow A_n$ as a
  conclusion of a syllogistic inference is by (ii), since exactly one
  bullet symbol must occur in the conclusion with two morphisms converging to it.
\item[(c)] the only way to obtain $A_1\leftarrow\b\arr A_n$ as
  a conclusion of a syllogistic inference is by (iii) or (iv), since
  exactly one bullet symbol must occur in the conclusion with two morphisms
  diverging from it.
\item[(d)] the only way to obtain $A_1\leftarrow\b\arr\b\leftarrow
  A_n$ as a conclusion of a syllogistic inference is by (v), (vi) or
  (vii), since exactly two bullet symbols must occur in the conclusion,
  with three alternating morphisms.
\end{itemize}
\end{proof}
\begin{lemma}\label{silemma}
For every positive natural number $n$, let $\vfi(n)$ and $\psi(n)$ be
the number of diagrams like those in points (vi) and (vii) of
lemma~\ref{nolemma}, respectively. The following facts hold
\begin{itemize}
\item[(i)] $\vfi(n)=\frac{(n-1)(n-2)}{2}$.
\item[(ii)] $\psi(n)=\frac{n(n-1)}{2}$.
\end{itemize}
\end{lemma}
\begin{proof}
\begin{itemize}
\item[(i)] 
For every positive natural number $n$, $\vfi(n+1)=\vfi(n)+(n-1)$. 
%
%% In order to see this,
%%   let us consider a few cases:
%% %
%% \begin{itemize}
%% %
%% \item[n=1:] $\vfi(1)=0$
%% %
%% \item[n=2:] $\vfi(2)=0=0+0=\vfi(1)+(1-1)$
%% %
%% \item[n=3:] $\vfi(3)=1$, since there can be only 
%% %
%% \[A_1\leftarrow\b\arr A_2\arr\b\leftarrow A_3\]
%% %
%% so that $1=0+1=\vfi(2)+(2-1)$.
%% %
%% \item[n=4:] $\vfi(4)=3$, since there can be only
%% %
%% \[A_1\leftarrow\b\arr A_2\arr\b\leftarrow A_3\leftarrow A_4\]
%% %
%% \[A_1\leftarrow A_2\leftarrow\b\arr A_3\arr\b\leftarrow A_4\]
%% %
%% \[A_1\leftarrow\b\arr A_2\arr A_3\arr\b\leftarrow A_4\]
%% %
%% so that $3=1+2=\vfi(3)+(3-1)$. 
%% %
%% \item[n=5:] $\vfi(5)=6$, since there can be only
%% %
%% \[A_1\leftarrow\b\arr A_2\arr\b\leftarrow A_3\leftarrow A_4\leftarrow A_5\]
%% %
%% \[A_1\leftarrow A_2\leftarrow\b\arr A_3\arr\b\leftarrow A_4\leftarrow A_5\]
%% %
%% \[A_1\leftarrow A_2\leftarrow A_3\leftarrow\b\arr A_4\arr\b\leftarrow
%% A_5\]
%% %
%% \[A_1\leftarrow\b\arr A_2\arr A_3\arr\b\leftarrow A_4\leftarrow A_5\]
%% %
%% \[A_1\leftarrow\b\arr A_2\arr A_3\arr A_4\arr\b\leftarrow A_5\]
%% %
%% \[A_1\leftarrow A_2\leftarrow\b\arr A_3\arr A_4\arr\b\leftarrow A_5\]
%% %
%% \end{itemize}
%
%%Thus, 
%
Because, passing from $n$ to $n+1$ is a matter of inserting one more
arrow symbol $\arr$ or $\leftarrow$, on the left, on the right or in
the middle of the diagrams constructed at $n$, so to extend
them with one more term-variable. 
There are exactly $n-1$ possibilities of doing this. 
Finally, by induction on the number of
term-variables, the thesis is easily achieved.
\item[(ii)] The argument is completely similar to the previous but for
  the fact that for every positive natural number $n$, $\psi(n+1)=\psi(n)+n$.
\end{itemize}
\end{proof}
\begin{theorem}\label{nsillo}
For every positive natural number $n$, an $n$-term syllogism is valid
if and only if there is a (necessarily unique) syllogistic inference
from its premisses to its conclusion. Moreover, the number of valid $n$-term
syllogisms is $3n^2-n$.
\end{theorem}
\begin{proof}
Lemma~\ref{nolemma} and lemma~\ref{silemma},
permit to conclude that the
$n$-term syllogisms in the table
\begin{eqnarray}\label{nvalsyll}
\begin{tabular}{|l|c|}
\hline
syllogism & quantity\\
\hline
$\catebf{A}_{A_{n-1}A_n},\ldots,\catebf{A}_{A_1A_2}\models\catebf{A}_{A_1A_n}$
& 1
\\
\hline
$\catebf{A}_{A_nA_{n-1}},\ldots,\catebf{E}_{A_iA_{i+1}},\ldots,\catebf{A}_{A_1A_2}\models\catebf{E}_{A_1A_n}$
& n-1\\
\hline
$\catebf{A}_{A_nA_{n-1}},\ldots,\catebf{E}_{A_{i+1}A_i},\ldots,\catebf{A}_{A_1A_2}\models\catebf{E}_{A_1A_n}$
& n-1\\
\hline
$\catebf{A}_{A_{n-1}A_n},\ldots,\catebf{I}_{A_iA_{i+1}},\ldots,\catebf{A}_{A_2A_1}\models\catebf{I}_{A_1A_n}$
& n-1\\
\hline
$\catebf{A}_{A_{n-1}A_n},\ldots,\catebf{I}_{A_{i+1}A_i},\ldots,\catebf{A}_{A_2A_1}\models\catebf{I}_{A_1A_n}$
& n-1\\
\hline
$\catebf{A}_{A_{n-1}A_n},\ldots,\catebf{I}_{A_iA_i},\ldots,\catebf{A}_{A_2A_1}\models\catebf{I}_{A_1A_n}$
& n\\
\hline
$\catebf{A}_{A_nA_{n-1}},\ldots,\catebf{O}_{A_iA_{i+1}},\ldots,\catebf{A}_{A_2A_1}\models\catebf{O}_{A_1A_n}$
& n-1\\
\hline
$\catebf{A}_{A_nA_{n-1}},\ldots,\catebf{E}_{A_{j-1}A_j},\ldots,\catebf{I}_{A_iA_{i+1}},\ldots,\catebf{A}_{A_2A_1}\models\catebf{O}_{A_1A_n}$
& $\frac{(n-1)(n-2)}{2}$\\
\hline
$\catebf{A}_{A_nA_{n-1}},\ldots,\catebf{E}_{A_{j-1}A_j},\ldots,\catebf{I}_{A_{i+1}A_i},\ldots,\catebf{A}_{A_2A_1}\models\catebf{O}_{A_1A_n}$
& $\frac{(n-1)(n-2)}{2}$\\
\hline
$\catebf{A}_{A_nA_{n-1}},\ldots,\catebf{E}_{A_jA_{j-1}},\ldots,\catebf{I}_{A_iA_{i+1}},\ldots,\catebf{A}_{A_2A_1}\models\catebf{O}_{A_1A_n}$
& $\frac{(n-1)(n-2)}{2}$\\
\hline
$\catebf{A}_{A_nA_{n-1}},\ldots,\catebf{E}_{A_jA_{j-1}},\ldots,\catebf{I}_{A_{i+1}A_i},\ldots,\catebf{A}_{A_2A_1}\models\catebf{O}_{A_1A_n}$
& $\frac{(n-1)(n-2)}{2}$\\
\hline
$\catebf{A}_{A_nA_{n-1}},\ldots,\catebf{E}_{A_{j-1}A_j},\ldots,\catebf{I}_{A_iA_i},\ldots,\catebf{A}_{A_2A_1}\models\catebf{O}_{A_1A_n}$
& $\frac{n(n-1)}{2}$\\
\hline
$\catebf{A}_{A_nA_{n-1}},\ldots,\catebf{E}_{A_jA_{j-1}},\ldots,\catebf{I}_{A_iA_i},\ldots,\catebf{A}_{A_2A_1}\models\catebf{O}_{A_1A_n}$
& $\frac{n(n-1)}{2}$\\
\hline
\end{tabular}
\end{eqnarray}
are all valid. Moreover they are $3n^2-n$. Conversely, we use lemma
~\ref{nolemma} and lemma~\ref{silemma}, to construct a syllogistic inference
to a given possible conclusion: 
\begin{itemize}
\item[-] By lemma~\ref{nolemma} (i), the diagram
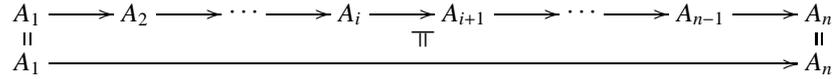
\[\xymatrix@R=1ex{A_1\ar@{}[drrrrrrr]|{\begin{turn}{270}$\models$\end{turn}}\ar@{=}[d]
\ar[r]&A_2\ar[r]&\cdots\ar[r]&A_i\ar[r]&A_{i+1}\ar[r]&\cdots\ar[r]&A_{n-1}\ar[r]&A_n\ar@{=}[d]\\
A_1\ar[rrrrrrr]&&&&&&&A_n}\]
represents the only way to produce evidence for the syllogistic
inference
\[(\catebf{A}_{A_{n-1}A_n})\sharp\cdots\sharp(\catebf{A}_{A_1A_2})\models(\catebf{A}_{A_1A_n})\]
validating the $n$-term syllogism 
\[\catebf{A}_{A_{n-1}A_n},\ldots,\catebf{A}_{A_1A_2}\models\catebf{A}_{A_1A_n}\]
\item[-] By lemma~\ref{nolemma} (ii), the $n-1$ diagrams
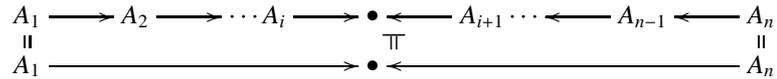
\[\xymatrix@R=1ex{A_1\ar@{}[drrrrrr]|{\begin{turn}{270}$\models$\end{turn}}\ar@{=}[d]
\ar[r]& A_2\ar[r]&\cdots A_i\ar[r]&\b&A_{i+1}\ar[l]
\cdots &A_{n-1}\ar[l]&A_n\ar[l]\ar@{=}[d]\\
A_1\ar[rrr]&&&\b&&&A_n\ar[lll]}\]
represent the only way to produce evidence for the syllogistic
inference 
\[(\catebf{A}_{A_nA_{n-1}})^{\circ}\sharp\cdots\sharp(\catebf{E}_{A_iA_{i+1}})
\sharp\cdots\sharp(\catebf{A}_{A_1A_2})\models(\catebf{E}_{A_1A_n})\]
as well as for the syllogistic inference
\[(\catebf{A}_{A_nA_{n-1}})^{\circ}\sharp\cdots\sharp(\catebf{E}_{A_{i+1}A_i})^{\circ}
\sharp\cdots\sharp(\catebf{A}_{A_1A_2})\models(\catebf{E}_{A_1A_n})\]
validating the $n-1$ syllogisms
\[\catebf{A}_{A_nA_{n-1}},\ldots,\catebf{E}_{A_iA_{i+1}},\ldots,\catebf{A}_{A_1A_2}
\models\catebf{E}_{A_1A_n}\]
and the $n-1$ syllogisms
\[\catebf{A}_{A_nA_{n-1}},\ldots,\catebf{E}_{A_{i+1}A_i},\ldots,\catebf{A}_{A_1A_2}
\models\catebf{E}_{A_1A_n}\]
respectively. Thus, in total 
there are $2(n-1)$ valid $n$-term syllogisms with conclusion
$\catebf{E}_{A_1A_n}$.
\item[-] By lemma~\ref{nolemma} (iii), the $n-1$ diagrams
\[\xymatrix@R=1ex{A_1\ar@{}[drrrrrr]|{\begin{turn}{270}$\models$\end{turn}}\ar@{=}[d]&A_2
\ar[l]&\cdots\ar[l]A_i&\b\ar[l]\ar[r]&
  A_{i+1}\cdots\ar[r]& A_{n-1}\ar[r]&A_n\ar@{=}[d]\\
A_1&&&\b\ar[lll]\ar[rrr]&&&A_n}\]
represent the only way to produce evidence for the $n-1$ syllogistic inferences
\[(\catebf{A}_{A_{n-1}A_n})\sharp\cdots\sharp(\catebf{I}_{A_iA_{i+1}})\sharp\cdots\sharp
(\catebf{A}_{A_2A_1})^{\circ}\models(\catebf{I}_{A_1A_n})\]
as well as for the $n-1$ syllogistic inferences
\[(\catebf{A}_{A_{n-1}A_n})\sharp\cdots\sharp(\catebf{I}_{A_{i+1}A_i})^{\circ}\sharp\cdots\sharp
(\catebf{A}_{A_2A_1})^{\circ}\models(\catebf{I}_{A_1A_n})\]
that validate the $n-1$ $n$-term syllogisms
\[\catebf{A}_{A_{n-1}A_n},\ldots,\catebf{I}_{A_iA_{i+1}},\ldots,\catebf{A}_{A_2A_1}
\models\catebf{I}_{A_1A_n}\]
and the $n-1$ $n$-term syllogisms
\[\catebf{A}_{A_{n-1}A_n},\ldots,\catebf{I}_{A_{i+1}A_i},\ldots,\catebf{A}_{A_2A_1}
\models\catebf{I}_{A_1A_n}\]
respectively.\\ 
By lemma~\ref{nolemma} (iv), the $n$ diagrams
\[\xymatrix@R=1ex{A_1\ar@{}[drrrrrr]|{\begin{turn}{270}$\models$\end{turn}}
\ar@{=}[d]&A_2\ar[l]&\cdots\ar[l]A_i&\b\ar[l]\ar[r]&
  A_i\cdots\ar[r]& A_{n-1}\ar[r]&A_n\ar@{=}[d]\\
A_1&&&\b\ar[lll]\ar[rrr]&&&A_n}\]
represent the only way to produce evidence for the $n$ syllogistic inferences
\[(\catebf{A}_{A_{n-1}A_n})\sharp\cdots\sharp(\catebf{I}_{A_iA_i})\sharp
\cdots\sharp(\catebf{A}_{A_2A_1})^{\circ}\models
(\catebf{I}_{A_1A_n})\]
that validate the $n$ $n$-term syllogisms
\[\catebf{A}_{A_{n-1}A_n},\ldots,\catebf{I}_{A_iA_i},\ldots,\catebf{A}_{A_2A_1}
\models\catebf{I}_{A_1A_n}\]
so that in total there are $2(n-1)+n$ valid $n$-term syllogisms with 
conclusion $\catebf{I}_{A_1A_n}$.
\item[-] By lemma~\ref{nolemma} (v), the $n-1$ diagrams
\[\xymatrix@R=1ex@C=3.5ex{A_1\ar@{=}[d]
\ar@{}[drrrrrrr]|{\begin{turn}{270}$\models$\end{turn}}
&A_2\ar[l]&{\cdots
  A_i}\ar[l]&\b\ar[l]\ar[r]&\b&A_{i+1}\ar[l]\cdots&
  A_{n-1}\ar[l]&A_n\ar[l]\ar@{=}[d]\\
A_1&&&\b\ar[lll]\ar[r]&\b&&&A_n\ar[lll]}\]
represent the only way to produce evidence for the $n-1$ syllogistic inferences
\[(\catebf{A}_{A_nA_{n-1}})^{\circ}\sharp\cdots\sharp(\catebf{O}_{A_iA_{i+1}})\sharp\cdots\sharp
(\catebf{A}_{A_2A_1})^{\circ}\models(\catebf{O}_{A_1A_n})\]
that validate the $n-1$ $n$-term syllogisms
\[\catebf{A}_{A_nA_{n-1}},\ldots,\catebf{O}_{A_iA_{i+1}},\ldots,\catebf{A}_{A_2A_1}
\models\catebf{O}_{A_1A_n}.\]
By lemma~\ref{nolemma} (vi) and lemma~\ref{silemma} (i), the $\frac{(n-1)(n-2)}{2}$ diagrams
\[\xymatrix@R=1ex@C=3.5ex{A_1\ar@{}[drrrrrrr]|{\begin{turn}{270}$\models$\end{turn}}
\ar@{=}[d]\ar@{<-}[r]&{\cdots
    A_i}&\b\ar[l]\ar[r]&A_{i+1}\cdots\ar[r]&\cdots A_{j-1}\ar[r]&
\b&{A_j\cdots}\ar[l]&A_n\ar[l]\ar@{=}[d]\\
A_1&&\b\ar[ll]\ar[rrr]&&&\b&&A_n\ar[ll]}\]
represent the only way to produce evidence for the
$4\cdot\frac{(n-1)(n-2)}{2}$ 
sylllogistic inferences
\[(\catebf{A}_{A_nA_{n-1}})^{\circ}\sharp\cdots\sharp(\catebf{E}_{A_{j-1}A_j})\sharp\cdots\sharp
(\catebf{I}_{A_iA_{i+1}})\sharp\cdots\sharp(\catebf{A}_{A_2A_1})^{\circ}
\models(\catebf{O}_{A_1A_n})\]
\[(\catebf{A}_{A_nA_{n-1}})^{\circ}\sharp\cdots\sharp(\catebf{E}_{A_{j-1}A_j})\sharp\cdots
(\catebf{I}_{A_{i+1}A_i})^{\circ}\sharp\cdots\sharp(\catebf{A}_{A_2A_1})^{\circ}
\models(\catebf{O}_{A_1A_n})\]
\[(\catebf{A}_{A_nA_{n-1}})^{\circ}\sharp\cdots\sharp(\catebf{E}_{A_jA_{j-1}})^{\circ}
\sharp\cdots\sharp
(\catebf{I}_{A_iA_{i+1}})\sharp\cdots\sharp(\catebf{A}_{A_2A_1})^{\circ}
\models(\catebf{O}_{A_1A_n})\]
\[(\catebf{A}_{A_nA_{n-1}})^{\circ}\sharp\cdots\sharp(\catebf{E}_{A_jA_{j-1}})^{\circ}
\sharp\cdots\sharp
(\catebf{I}_{A_{i+1}A_i})^{\circ}\sharp\cdots\sharp(\catebf{A}_{A_2A_1})^{\circ}
\models(\catebf{O}_{A_1A_n})\]
that validate the $4\cdot\frac{(n-1)(n-2)}{2}$ $n$-term syllogisms
\[\catebf{A}_{A_nA_{n-1}},\ldots,\catebf{E}_{A_{j-1}A_j},\ldots,\catebf{I}_{A_iA_{i+1}},
\ldots,\catebf{A}_{A_2A_1}\models\catebf{O}_{A_1A_n}\]
\[\catebf{A}_{A_nA_{n-1}},\ldots,\catebf{E}_{A_{j-1}A_j},\ldots,\catebf{I}_{A_{i+1}A_i},
\ldots,\catebf{A}_{A_2A_1}\models\catebf{O}_{A_1A_n}\]
\[\catebf{A}_{A_nA_{n-1}},\ldots,\catebf{E}_{A_jA_{j-1}},\ldots,\catebf{I}_{A_iA_{i+1}},
\ldots,\catebf{A}_{A_2A_1}\models\catebf{O}_{A_1A_n}\]
\[\catebf{A}_{A_nA_{n-1}},\ldots,\catebf{E}_{A_jA_{j-1}},\ldots,\catebf{I}_{A_{i+1}A_i},
\ldots,\catebf{A}_{A_2A_1}\models\catebf{O}_{A_1A_n}\]
respectively.\\ 
By lemma~\ref{nolemma} (vii) and lemma~\ref{silemma} (ii),
the $\frac{n(n-1)}{2}$ diagrams
\[\xymatrix@R=1ex@C=3.5ex{A_1\ar@{}[drrrrrrr]|{\begin{turn}{270}$\models$\end{turn}}
\ar@{=}[d]\ar@{<-}[r]&{\cdots
    A_i}&\b\ar[l]\ar[r]&A_i\cdots\ar[r]&\cdots A_{j-1}\ar[r]&
\b&{A_j\cdots}\ar[l]&A_n\ar[l]\ar@{=}[d]\\
A_1&&\b\ar[ll]\ar[rrr]&&&\b&&A_n\ar[ll]}\]
represent the only way to produce evidence for the $2\cdot\frac{n(n-1)}{2}$
syllogistic inferences
\[(\catebf{A}_{A_nA_{n-1}})^{\circ}\sharp\cdots\sharp(\catebf{E}_{A_{j-1}A_j})\sharp\cdots\sharp
(\catebf{I}_{A_iA_i})\sharp\cdots\sharp(\catebf{A}_{A_2A_1})^{\circ}\models
(\catebf{O}_{A_1A_n})\]
\[(\catebf{A}_{A_nA_{n-1}})^{\circ}\sharp\cdots\sharp(\catebf{E}_{A_jA_{j-1}})^{\circ}
\sharp\cdots\sharp
(\catebf{I}_{A_iA_i})\sharp\cdots\sharp(\catebf{A}_{A_2A_1})^{\circ}
\models(\catebf{O}_{A_1A_n})\]
that validate the $2\cdot\frac{n(n-1)}{2}$ $n$-term syllogisms
\[\catebf{A}_{A_nA_{n-1}},\ldots,\catebf{E}_{A_{j-1}A_j},\ldots,\catebf{I}_{A_iA_i},
\ldots,\catebf{A}_{A_2A_1}\models\catebf{O}_{A_1A_n}\]
\[\catebf{A}_{A_nA_{n-1}},\ldots,\catebf{E}_{A_jA_{j-1}},\ldots,\catebf{I}_{A_iA_i},
\ldots,\catebf{A}_{A_2A_1}\models\catebf{O}_{A_1A_n}\]
respectively. Thus in total there are
$n-1+4\cdot\frac{(n-1)(n-2)}{2}+2\cdot\frac{n(n-1)}{2}$ valid $n$-term
syllogisms with conclusion $\catebf{O}_{A_1A_n}$.
\end{itemize}
In total, the valid $n$-term syllogisms are in number of
\[1+2(n-1)+2(n-1)+n+(n-1)+4\cdot\frac{(n-1)(n-2)}{2}+2\cdot\frac{n(n-1)}{2}=3n^2-n\]
\end{proof}
We end the section with the explicit description 
of the valid $n$-term syllogisms
for $n=1$ and $n=2$, respectively. In the first case, there is only
one figure, that is $A_1A_1$ and only two valid moods for it, that is
$\catebf{A}$ and $\catebf{I}$ so that, as observed
in~\cite{Lukasiewicz} and~\cite{Meredith}, the only valid $1$-term syllogisms are
$\catebf{A}_{A_1A_1}\models\catebf{A}_{A_1A_1}$ and 
$\catebf{I}_{A_1A_1}\models\catebf{I}_{A_1A_1}$, that is the laws of
identity to which we hinted at in the previous section. In the
second case there are two figures, as shown in the table
\[\begin{tabular}{|l|c|c|}
\hline
& fig. 1 & fig. 2\\
\hline
premise & $A_1A_2$ & $A_1A_2$\\
\hline
conclusion & $A_1A_2$ & $A_2A_1$\\
\hline
\end{tabular}\]
and ten valid $2$-term syllogisms, six in the first figure and four
in the second, as follows:
\begin{description}
\item[figure 1:] $\catebf{A}_{A_1A_2}\models\catebf{A}_{A_1A_2}$,
  $\catebf{E}_{A_1A_2}\models\catebf{E}_{A_1A_2}$,
  $\catebf{I}_{A_1A_2}\models\catebf{I}_{A_1A_2}$,
  $\catebf{O}_{A_1A_2}\models\catebf{O}_{A_1A_2}$, 
plus the \emph{laws of subalternation}
  $\catebf{A}_{A_1A_2},\catebf{I}_{A_1A_1}\models\catebf{I}_{A_1A_2}$, $\catebf{E}_{A_1A_2},\catebf{I}_{A_1A_1}\models\catebf{O}_{A_1A_2}$.
\item[figure 2:] $\catebf{E}_{A_2A_1}\models\catebf{E}_{A_1A_2}$,
  $\catebf{I}_{A_2A_1}\models\catebf{I}_{A_1A_2}$ which are the
  \emph{laws of simple conversion}, and
  $\catebf{I}_{A_2A_2},\catebf{A}_{A_2A_1}\models\catebf{I}_{A_1A_2}$,
  $\catebf{E}_{A_2A_1},\catebf{I}_{A_1A_1},\models\catebf{O}_{A_1A_2}$
  which are the \emph{laws of conversion per accidens}.
\end{description}
\section{Syllogisms as rewrite rules}\label{just}
We here point out the existing connections between the previously
introduced calculus of syllogisms and the rewriting of certain terms,
on the base of suitable rewrite rules. On term rewriting systems in
general, the reader may consult~\cite{MR1629216}. Following this, 
we look at rewrite rules as to directed equations, separating a
reducible expression on their left-hand side from a reduced one on
their right-hand side, and at term rewriting as to a computation
mechanism. Applying a rewrite rule gives rise to a
\emph{reduction}. Loosely, from our standpoint the terms and the rewrite rules we are
interested in are finite sequences of Aristotelian diagrams and the valid syllogisms,
respectively, whereas reductions are finite pastings of syllogistic
inferences, so that $\models$
extends to a \emph{reduction relation} on the set of finite sequences
of Aristotelian diagrams. In
pursuing this point of view, we will single out a suitable
category theoretic framework for the calculus. Because of this,
we assume that the reader has already knowledge of the basics 
in category theory as can be found in~\cite{MR1712872}, for
example.\\

We recall that a term is in \emph{normal form} or \emph{irreducible} 
if no rewrite rule applies to it. Otherwise, it is \emph{reducible}. 
If $\models$ denotes a reduction relation, then 
repeated applications of rewrite rules to a reducible
term $t_1$ yield a descending chain of reductions
\[t_1\models t_2\models t_3\models\cdots\]
which may not be finite. A finite chain of reductions will be more
briefly denoted $t_1\models^* t_n$ and a term $t$ is a \emph{direct
  successor} of a term $s$ if $s\models t$.\\
A term rewriting system is 
\begin{itemize}
\item[-] \emph{normalizing} if every term reduces to a normal form.
\item[-] \emph{terminating} if there is no infinite desceding chain of
  reductions
\[t_1\models t_2\models t_3\models\cdots\]
\item[-] \emph{locally confluent} if in reason of the application of
  two different rewrite rules to a term $s$, yielding in turn two
  different terms $t_1$ and $t_2$, then a term $t$ and two finite chain of reductions
\[t_1\models^* t\qquad t_2\models^* t\]
exist.
\item[-] \emph{confluent} if whenever $s\models^* t_1$ and $s\models^*
  t_2$, with $t_1$ and $t_2$ different terms, then 
there exists a term $t$ and two finite chains of reductions
\[t_1\models^* t\qquad t_2\models^* t\]
\item[-] \emph{convergent} if it is both terminating and confluent.
\end{itemize}
A terminating rewriting system is normalizing but not the
converse, see~\cite{MR1629216}, so that in a terminating rewriting
system every term has at least one normal form. On the other hand in a confluent
rewriting system every term has at most one normal form. Thus, in a
convergent rewriting system every term has exactly one normal form.\\
Now, for future reference we mention the following
\begin{lemma}[Newman's lemma]\label{Newman}
A terminating and locally confluent term rewriting system is confluent.
\end{lemma}
\begin{proof}
See~\cite{MR1629216}.
\end{proof}

For every natural number $n$, $n$-polygraphs and more in general 
$\infty$-polygraphs were introduced in~\cite{MR1224519}. Existing
connections between $n$-polygraphs, for $n=2,3$ in particular, and
rewriting systems were pointed out in~\cite{MR2254890} and~\cite{GuiraudMalbos},
for example, so that later on we will feel free to extend to them the previously introduced
terminology.\\
Now, the idea is that of looking at the calculus of $n$-term
syllogisms as to a rewriting system, or better as to a specific
$2$-polygraph. In this way, it turns out that the calculus takes place in a suitable
categorical structure, freely generated by such a $2$-polygraph.\\

From~\cite{MR1224519}, with personal notations,  we
briefly recall that for every natural number $n$, an $n$-\emph{graph}
is a diagram of sets and functions
\[\xymatrix{G_0&G_1\ar@<-.8ex>[l]_(.4){s_0}\ar@<.8ex>[l]^(.4){t_0}&G_2
\ar@<-.8ex>[l]_(.4){s_1}\ar@<.8ex>[l]^(.4){t_1}&\cdots&
  G_{n-1}&G_n\ar@<.8ex>[l]^(.4){t_{n-1}}\ar@<-.8ex>[l]_(.4){s_{n-1}}}\]
such that for every positive natural number $n$, the \emph{globular identities}
$s_{n-1}s_n=s_{n-1}t_n$ and $t_{n-1}s_n=t_{n-1}t_n$ hold. For every
$0\leq i\leq n$, the functions
$s_i$ and $t_i$ are referred to as \emph{source}
and \emph{target}, respectively, whereas the elements of $G_i$ are
referred to as $i$-\emph{cells}. A $0$-graph is just a set, whereas a $1$-graph is
an ordinary graph. A \emph{morphism of $n$-graphs}, is a family of $n$
functions carrying $i$-cells to $i$-cells, commuting with 
source and target.\\

A $2$-category is a $2$-graph 
\begin{eqnarray}\label{2cat}
\xymatrix{G_0&G_1\ar@<.8ex>[l]^(.4){t_0}\ar@<-.8ex>[l]_(.4){s_0}&G_2
\ar@<.8ex>[l]^(.4){t_1}\ar@<-.8ex>[l]_(.4){s_1}}
\end{eqnarray}
equipped with a category structure on the graph $(G_0,G_1,s_0,t_0)$, a
category structure on the graph $(G_0,G_2,s_0s_1,t_0t_1)$ 
and a category structure on the
graph $(G_1,G_2,s_1,t_1)$, reciprocally compatible. 
In explicit and elementary terms, 
a $2$-category consists of objects and morphisms, also respectively called 
$0$-cells and $1$-cells, that identify a category and moreover 
of $2$-cells $\alpha, \beta, \gamma, \ldots$ between parallel pairs
of $1$-cells such that to every 
morphism $f:A\arr B$ corresponds a designated identity $2$-cell $id_f:f\Arr f$, and
to every morphisms $f,g,h:A\arr B$ and $2$-cells $\alpha:f\Arr g$ and
$\beta:g\Arr h$ corresponds a designated \emph{vertical composite} 
$\beta\cdot\alpha:f\Arr h$, in such a way that the axioms for a
category are satisfied. 
Often, the identity $2$-cell associated to a
$1$-cell $f$ is denoted by the same letter $f$, thus writing 
$f:f\Arr f$. For $1$-cells $f,g:A\arr
B$, $f',g':B\arr C$, to every $2$-cells $\alpha:f\Arr g$ and
$\alpha':f'\Arr g'$ there corresponds a designated \emph{horizontal composite}
$\alpha'\ast\alpha:f'\circ f\Arr g'\circ g$, such that $\alpha'\ast
id_B=\alpha'=id_C\ast\alpha'$. Also, $f'\ast
f=f'\circ f$. Finally, vertical and horizontal composition of
$2$-cells are required to interact suitably, so that 
for every further morphisms $h:A\arr B$, $h':B\arr C$ and for every $2$-cells
$\beta:g\Arr h$ and $\beta':g'\Arr h'$, the \emph{interchange law} 
$(\beta'\cdot\alpha')\ast(\beta\cdot\alpha)=(\beta'\ast\beta)\cdot(\alpha'\ast\alpha)$
holds. A morphism between two $2$-categories 
is a morphism between the underlying $2$-graphs, furthermore
preserving horizontal and vertical compositions and identities, 
and usually referred to as $2$-\emph{functor}.
\begin{example}
Categories, functors and natural transformations considered as
$0$-cells, $1$-cells and $2$-cells respectively, identify the
paradigmatic example of $2$-category.
\end{example}
\begin{example}\label{locdiscr}
Every ordinary category can be seen as a $2$-category with only
identical $2$-cells, which is also called \emph{locally discrete}.
\end{example}
\begin{example}
Sets and relations form a category which  
can be seen as a $2$-category
by letting the $2$-cells be inclusions, i.e. for sets $A$, $B$ and relations $R,S\subseteq
A\times B$, there is a $2$-cell $R\Arr S$ if and only if $R\subseteq S$.
\end{example}
Now, following~\cite{MR1224519} and with personal notations, we give the following
\begin{definition}
A $2$-polygraph $\Sigma$ consists of a diagram 
\begin{eqnarray}\label{2poly}
\xymatrix{&&\Sigma_1\ar@{^{(}->}[d]
\ar@<.6ex>[dll]|{s_0}\ar@<-.8ex>[dll]|{t_0}&&
\Sigma_2\ar@<.6ex>[dll]|{s_1}\ar@<-.8ex>[dll]|{t_1}\\
\Sigma_0&&\Sigma_1^*\ar@<.5ex>[ll]^(.3){\ov{s_0}}\ar@<-.5ex>[ll]_(.3){\ov{t_0}}}
\end{eqnarray}
in which, the functions $\ov{s_0}$, $\ov{t_0}$ 
are the source and target functions of the free 
category generated by the graph $(\Sigma_0,\Sigma_1,s_0,t_0)$,
and where
\begin{eqnarray}\label{2graphund}
\xymatrix{\Sigma_0&\Sigma_1^*\ar@<.8ex>[l]^{\ov{s_0}}\ar@<-.8ex>[l]_{\ov{t_0}}&
\Sigma_2\ar@<.8ex>[l]^{s_1}\ar@<-.8ex>[l]_{t_1}}
\end{eqnarray}
is a $2$-graph.
\end{definition}

As observed in~\cite{MR1224519}, $2$-polygraphs are in connection with
the study of the word problem in categories, generalizing the
one in monoids. The \emph{free $2$-category generated by a
  $2$-polygraph} $\Sigma$, is the $2$-category $\Sigma^*$ occurring in the 
lower part of diagram
\begin{eqnarray}\label{lower}
\xymatrix{&&\Sigma_1\ar@{^{(}->}[d]
\ar@<.6ex>[dll]|{s_0}\ar@<-.8ex>[dll]|{t_0}&&
\Sigma_2\ar@<.6ex>[dll]|{s_1}\ar@<-.8ex>[dll]|{t_1}\ar@{^{(}->}[d]\\
\Sigma_0&&\Sigma_1^*\ar@<.5ex>[ll]^(.3){\ov{s_0}}\ar@<-.5ex>[ll]_(.3){\ov{t_0}}&&
\Sigma_2^*\ar@<.5ex>[ll]^(.3){\ov{s_1}}\ar@<-.5ex>[ll]_(.3){\ov{t_1}}}
\end{eqnarray}
that is the free $2$-category generated by the category
$\xymatrix{\Sigma_0&\Sigma_1^*\ar@<.8ex>[l]^(.4){\ov{s_0}}\ar@<-.8ex>[l]_(.4){\ov{t_0}}}$
with additional $2$-cells provided by $\Sigma_2$. By
following~\cite{MR1224519} again, a more explicit description of such
a free $2$-category can be given by looking at the elements of $\Sigma_2$ as to
diagrams of the form
\[\xymatrix@R=1ex{X\ar@{=}[d]\ar@{}[drrr]|{\begin{turn}{270}$\Arr$\end{turn}}
\ar[r]^{f_1}&\cdots\ar[r]&\cdots\ar[r]^{f_n}&Y\ar@{=}[d]\\
X\ar[r]_{g_1}&\cdots\ar[r]&\cdots\ar[r]_{g_m}&Y}\]
and at the elements of $\Sigma_2^*$ as to $2$-paths.\\
There exists a morphism of $2$-graphs from the
$2$-graph~(\ref{2graphund}) to the $2$-graph underlying the $2$-category
which is the lower part of diagram~(\ref{lower}), precisely the one which is the inclusion
of $\Sigma_2$ in $\Sigma_2^*$ and the identical function on $\Sigma_0$
and $\Sigma_1^*$, so that freeness of the $2$-category 
\[\xymatrix{\Sigma_0&\Sigma_1^*\ar@<.8ex>[l]^{\ov{s_0}}\ar@<-.8ex>[l]_{\ov{t_0}}&
\Sigma_2^*\ar@<.8ex>[l]^{\ov{s_1}}\ar@<-.8ex>[l]_{\ov{t_1}}}\]
amounts to the following: for every
$2$-category $\cateb{C}=(C_0,C_1,C_2)$ and for every
morphism of $2$-graphs $F=(F_0,F_1,F_2):(\Sigma_0,\Sigma_1^*,\Sigma_2)\arr\cateb{C}$
there exists a unique $2$-functor $F^*=(F_0,F_1,F_2^*):\Sigma^*\arr\cateb{C}$
extending $F$, namely such that the diagram
\[\xymatrix{\Sigma_2\ar[dr]_{F_2}\ar@{^{(}->}[r]&\Sigma_2^*\ar[d]^{F_2^*}\\
&C_2}\]
commutes.
\begin{definition}\label{polycalc}
Let $n$ be a positive natural number. 
The \emph{polygraph for the calculus of $n$-term syllogisms} is the
$2$-polygraph \cate{S} identified by the following data:
\begin{itemize}
\item[-] a set $\cate{S}_0=\{A_1,\ldots,A_n\}$ of \emph{term-variables}.
\item[-] a set $\cate{S}_1$ whose elements are the Aristotelian diagrams
\[\begin{tabular}{l l l}
$\AA{A_i}{A_j}$ & $\AAo{A_j}{A_i}$ & $1\leq i\leq j\leq n$\\
$\EE{A_i}{A_j}$ & $\EEo{A_j}{A_i}$ & $1\leq i\leq j\leq n$\\
$\II{A_i}{A_j}$ & $\IIo{A_j}{A_i}$ & $1\leq i\leq j\leq n$\\
$\OO{A_i}{A_j}$ & $\OOo{A_j}{A_i}$ & $1\leq i\leq j\leq n$
\end{tabular}\]
together with the evident source and target functions.
\item[-] a set $\cate{S}_2$ of \emph{rewrite rules}, which are the
  following syllogistic inferences for the corresponding valid syllogisms:
\[\begin{tabular}{c c c}
$(\aa{A_i}{A_i})\models(\aa{A_i}{A_i})$ & 
  $(\ii{A_i}{A_i})\models(\ii{A_i}{A_i})$ & $1\leq i \leq n$\\
\end{tabular}\]
\[\begin{tabular}{c c c}
$(\aa{A_i}{A_j})\models(\aa{A_i}{A_j})$ & 
  $(\ii{A_i}{A_j})\models(\ii{A_i}{A_j})$ & $1\leq i\lt j\leq n$\\
$(\ee{A_i}{A_j})\models(\ee{A_i}{A_j})$ &
$(\oo{A_i}{A_j})\models(\oo{A_i}{A_j})$ & $1\leq i\lt j\leq n$\\
$(\aa{A_i}{A_j})\sharp(\ii{A_i}{A_i})\models\ii{A_i}{A_j}$ &
$(\ee{A_i}{A_j})\sharp(\ii{A_i}{A_i})\models\oo{A_i}{A_j}$ & $1\leq
i\lt j\leq n$
\end{tabular}\]
\[\begin{tabular}{c c c}
$(\ee{A_j}{A_i})^{\circ}\models(\ee{A_i}{A_j})$ &
$(\ii{A_j}{A_i})^{\circ}\models(\ii{A_i}{A_j})$ & $1\leq i\lt j\leq n$\\
$(\ii{A_j}{A_j})\sharp(\aa{A_j}{A_i})^{\circ}\models(\ii{A_i}{A_j})$ &
$(\ee{A_j}{A_i})^{\circ}\sharp(\ii{A_i}{A_i})\models(\oo{A_i}{A_j})$ &
$1\leq i\lt j\leq n$
\end{tabular}\]
\[\begin{tabular}{c c c}
$(\aa{A_j}{A_k})\sharp(\aa{A_i}{A_j})\models(\aa{A_i}{A_k})$ &
$(\ee{A_j}{A_k})\sharp(\aa{A_i}{A_j})\models(\ee{A_i}{A_k})$ & $1\leq
i\lt j\lt k\leq n$\\
$(\catebf{A}_{A_jA_k})\sharp(\catebf{I}_{A_iA_j})\models(\catebf{I}_{A_iA_k})$&
$(\catebf{E}_{A_jA_k})\sharp(\catebf{I}_{A_iA_j})\models(\catebf{O}_{A_iA_k})$&
$1\leq i\lt j\lt k\leq n$
\end{tabular}\]
\[\begin{tabular}{c c c}
$(\catebf{E}_{A_kA_j})^{\circ}\sharp(\catebf{A}_{A_iA_j})\models(\catebf{E}_{A_iA_k})$&
$(\catebf{A}_{A_kA_j})^{\circ}\sharp(\catebf{E}_{A_iA_j})\models(\catebf{E}_{A_iA_k})$&
$1\leq i\lt j\lt k\leq n$\\
$(\catebf{E}_{A_kA_j})^{\circ}\sharp(\catebf{I}_{A_iA_j})\models(\catebf{O}_{A_iA_k})$&
$(\catebf{A}_{A_kA_j})^{\circ}\sharp(\catebf{O}_{A_iA_j})\models(\catebf{O}_{A_iA_k})$&
$1\leq i\lt j\lt k\leq n$
\end{tabular}\]
\[\begin{tabular}{c c c}
$(\catebf{I}_{A_jA_k})\sharp(\catebf{A}_{A_jA_i})^{\circ}\models(\catebf{I}_{A_iA_k})$&
$(\catebf{A}_{A_jA_k})\sharp(\catebf{I}_{A_jA_i})^{\circ}\models(\catebf{I}_{A_iA_k})$&$1\leq
i\lt j\lt k\leq n$\\
$(\catebf{O}_{A_jA_k})\sharp(\catebf{A}_{A_jA_i})^{\circ}\models(\catebf{O}_{A_iA_k})$&
$(\catebf{E}_{A_jA_k})\sharp(\catebf{I}_{A_jA_i})^{\circ}\models(\catebf{O}_{A_iA_k})$&
$1\leq i\lt j\lt k\leq n$
\end{tabular}\]
\[\begin{tabular}{c c c}
$(\catebf{A}_{A_kA_j})^{\circ}\sharp(\catebf{E}_{A_jA_i})^{\circ}\models(\catebf{E}_{A_iA_k})$&
$(\catebf{I}_{A_kA_j})^{\circ}\sharp(\catebf{A}_{A_jA_i})^{\circ}\models(\catebf{I}_{A_iA_k})$&
$1\leq i\lt j\lt k\leq n$\\
$(\catebf{E}_{A_kA_j})^{\circ}\sharp(\catebf{I}_{A_jA_i})^{\circ}\models(\catebf{O}_{A_iA_k})$& 
&$1\leq i\lt j\lt k\leq n$
\end{tabular}\]
%
%% \[\begin{tabular}{c c}
%% %
%% $(\catebf{A}_{A_jA_k})\sharp(\catebf{A}_{A_iA_j})\sharp(\catebf{I}_{A_iA_i})
%%\models(\catebf{I}_{A_iA_k})$&$1\leq
%% i\lt j \lt k$\\
%% %
%% $(\catebf{E}_{A_jA_k})\sharp(\catebf{A}_{A_iA_j})\sharp(\catebf{I}_{A_iA_i})
%%\models(\catebf{O}_{A_iA_k})$&$1\leq
%% i\lt j \lt k$
%% %
%% \end{tabular}\]
%% %
%% \[\begin{tabular}{c c}
%% %
%% $(\catebf{A}_{A_kA_j})^{\circ}\sharp(\catebf{E}_{A_iA_j})\sharp(\catebf{I}_{A_iA_i})
%%\models(\catebf{O}_{A_iA_k})$&$1\leq
%% i\lt j \lt k$\\
%% %
%% $(\catebf{E}_{A_kA_j})^{\circ}\sharp(\catebf{A}_{A_iA_j})\sharp(\catebf{I}_{A_iA_i})
%%\models(\catebf{O}_{A_iA_k})$&$1\leq
%% i\lt j \lt k$
%% %
%% \end{tabular}\]
%% %
%% \[\begin{tabular}{c c}
%% %
%% $(\catebf{A}_{A_jA_k})\sharp(\catebf{I}_{A_jA_j})\sharp(\catebf{A}_{A_jA_i})^{\circ}
%%\models(\catebf{I}_{A_iA_k})$&$1\leq
%% i\lt j \lt k$\\
%% %
%% $(\catebf{E}_{A_jA_k})\sharp(\catebf{I}_{A_jA_j})\sharp(\catebf{A}_{A_jA_i})^{\circ}
%%\models(\catebf{O}_{A_iA_k})$&$1\leq
%% i\lt j \lt k$
%% %
%% \end{tabular}\]
%% %
%% \[\begin{tabular}{c c}
%% %
%% $(\catebf{A}_{A_kA_j})^{\circ}\sharp(\catebf{E}_{A_jA_i})^{\circ}\sharp(\catebf{I}_{A_iA_i})
%%\models(\catebf{O}_{A_iA_k})$&$1\leq
%% i\lt j \lt k$\\
%% %
%% $(\catebf{E}_{A_kA_j})^{\circ}\sharp(\catebf{I}_{A_jA_j})\sharp(\catebf{A}_{A_jA_i})^{\circ}
%% \models(\catebf{O}_{A_iA_k})$&$1\leq i\lt j \lt k$\\
%% %
%% $(\catebf{I}_{A_kA_k})\sharp(\catebf{A}_{A_kA_j})^{\circ}\sharp(\catebf{A}_{A_jA_i})^{\circ}
%%\models(\catebf{I}_{A_iA_k})$&$1\leq i\lt j \lt k$
%% %
%% \end{tabular}\]
%
together with the evident source and target functions.
\end{itemize}
\end{definition}

It is useful to exemplify the previous
definition in the cases $n=1,2,3$.\\

If $n=1$, then 
\begin{itemize}
\item[-]
$\cate{S}_0=\{A_1\}$. 
\item[-]
$\cate{S}_1=\{\aa{A_1}{A_1},(\aa{A_1}{A_1})^{\circ},\ee{A_1}{A_1},(\ee{A_1}{A_1})^{\circ},
\ii{A_1}{A_1},(\ii{A_1}{A_1})^{\circ},\oo{A_1}{A_1},(\oo{A_1}{A_1})^{\circ}\}$.
\item[-]
  $\cate{S}_2=\{(\aa{A_1}{A_1})\models(\aa{A_1}{A_1}),(\ii{A_1}{A_1})
\models(\ii{A_1}{A_1})\}$.
\end{itemize}
so that $\cate{S}^*$ is a locally discrete $2$-category, see
example~\ref{locdiscr}, that is an ordinary category. 
The calculus of $1$-term syllogisms consists of the sole laws of
identity $\aa{A_1}{A_1}\models\aa{A_1}{A_1}$ and $\ii{A_1}{A_1}\models\ii{A_1}{A_1}$,
recovered by the rewrite rules in $\cate{S}_2$ above.\\

If $n=2$, then 
\begin{itemize}
\item[-] $\cate{S}_0=\{A_1,A_2\}$.
\item[-] $\cate{S}_1=\{(\catebf{X}_{A_iA_i})\,|\, i=1,2\}
\cup\{(\catebf{X}_{A_iA_i})^{\circ}\,|\, i=1,2\}
\cup\{(\catebf{X}_{A_iA_j})\,|\, 1\leq i\lt j\leq
  2\}\cup\cup\{(\catebf{X}_{A_jA_i})^{\circ}\,|\, 1\leq i\lt j\leq
  2\}$, where $\catebf{X}\in\{\catebf{A},\catebf{E},\catebf{I},\catebf{O}\}$.
\item[-] $\cate{S}_2$ consists of the rewrite rules
\[\begin{tabular}{c c c}
$(\aa{A_i}{A_i})\models(\aa{A_i}{A_i})$ & 
  $(\ii{A_i}{A_i})\models(\ii{A_i}{A_i})$ & $1\leq i \leq 2$\\
\end{tabular}\]
\[\begin{tabular}{c c c}
$(\aa{A_i}{A_j})\models(\aa{A_i}{A_j})$ & 
  $(\ii{A_i}{A_j})\models(\ii{A_i}{A_j})$ & $1\leq i\lt j\leq 2$\\
$(\ee{A_i}{A_j})\models(\ee{A_i}{A_j})$ &
$(\oo{A_i}{A_j})\models(\oo{A_i}{A_j})$ & $1\leq i\lt j\leq 2$\\
$(\aa{A_i}{A_j})\sharp(\ii{A_i}{A_i})\models\ii{A_i}{A_j}$ &
$(\ee{A_i}{A_j})\sharp(\ii{A_i}{A_i})\models\oo{A_i}{A_j}$ & $1\leq
i\lt j\leq 2$
\end{tabular}\]
\[\begin{tabular}{c c c}
$(\ee{A_j}{A_i})^{\circ}\models(\ee{A_i}{A_j})$ &
$(\ii{A_j}{A_i})^{\circ}\models(\ii{A_i}{A_j})$ & $1\leq i\lt j\leq 2$\\
$(\ii{A_j}{A_j})\sharp(\aa{A_j}{A_i})^{\circ}\models(\ii{A_i}{A_j})$ &
$(\ee{A_j}{A_i})^{\circ}\sharp(\ii{A_i}{A_i})\models(\oo{A_i}{A_j})$ &
$1\leq i\lt j\leq 2$
\end{tabular}\]
\end{itemize}
The previous data extend those for the calculus of $1$-term
syllogisms to recover the calculus of $2$-term
syllogisms. In particular, it is
possible to recognize the syllogistic inferences validating the
laws of subalternation, simple conversion
and conversion per accidens, already encountered at the end of section~\ref{nsyll}.\\ 

If $n=3$, then 
the data for the calculus of $3$-term syllogisms extend the previous
and amount to the whole of those listed in definition~\ref{polycalc}. 
The syllogistic inferences validating the syllogisms with
assumption of existence in table~(\ref{valhypsyll}) are
obtainable. For instance, the syllogistic inference validating $\catebf{AAI}$ in the fourth
figure and \catebf{AEO} in the second figure are given by the two step reductions
\[(\ii{A_3}{A_3})\sharp(\aa{A_3}{A_2})^{\circ}\sharp(\aa{A_2}{A_1})^{\circ}
\models(\ii{A_2}{A_3})\sharp(\aa{A_2}{A_1})^{\circ}\models(\ii{A_1}{A_3})\]
and
\[(\aa{A_3}{A_2})^{\circ}\sharp(\ee{A_1}{A_2})\sharp(\ii{A_1}{A_1})
\models(\aa{A_3}{A_2})^{\circ}\sharp(\oo{A_1}{A_2})\models(\oo{A_1}{A_3})\]
respectively, where the evident rewrite rules have been applied. The
calculations for the remaining syllogisms with assumption of existence
are similar.\\

The elements of $\cate{S}_1^*$ are words in the elements of $\cate{S}_1$
and will be henceforth referred to as \emph{terms}. 
We let the \emph{length}
of a term be the number of Aristotelian diagrams by which it
is formed, to which we will also refer to as \emph{premisses}. Thus
for example $(\aa{A_2}{A_3})$ is a term of length $1$,
whereas $(\ee{A_3}{A_4})^{\circ}\sharp(\ii{A_2}{A_3})\sharp(\aa{A_2}{A_1})^{\circ}$
is a term of length $3$. It is intuitively clear what a \emph{subterm}
is: $(\ee{A_3}{A_4})^{\circ}$,
$(\ee{A_3}{A_4})^{\circ}\sharp(\ii{A_2}{A_3})$ are substerms of
$(\ee{A_3}{A_4})^{\circ}\sharp(\ii{A_2}{A_3})\sharp(\aa{A_2}{A_1})^{\circ}$,
for example. Moreover, \emph{overlapping subterms} may occurr, namely
those that have some of their premisses in common, so that for example
$(\ee{A_3}{A_4})^{\circ}$ and 
$(\ee{A_3}{A_4})^{\circ}\sharp(\ii{A_2}{A_3})$ are overlapping
subterms of
$(\ee{A_3}{A_4})^{\circ}\sharp(\ii{A_2}{A_3})\sharp(\aa{A_2}{A_1})^{\circ}$,
as well as $(\ee{A_3}{A_4})^{\circ}\sharp(\ii{A_2}{A_3})$ and 
$(\ii{A_2}{A_3})\sharp(\aa{A_2}{A_1})^{\circ}$.
Terms undergo reduction by the rewrite rules
in $\cate{S}_2$. These are said to be trivial if they have exactly one
premise coinciding with their conclusion, otherwise are non-trivial. In
the free $2$-category $\cate{S}^*$, the trivial rewrite rules will
correspond to identical $2$-cells.
\begin{example}\label{exxe} 
\begin{itemize}
\item[(i)] the term
\[(\ee{A_1}{A_2})\sharp(\aa{A_1}{A_1})^{\circ}\]
cannot be rewritten on the base of any of the rewrite rules in $\cate{S}_2$.
\item[(ii)] the term
\[(\ee{A_3}{A_4})\sharp(\ii{A_3}{A_3})\sharp(\ee{A_2}{A_3})\sharp(\ii{A_1}{A_2})\]
undergoes rewriting by the sequential application of the rewrite
rules 
\[(\ee{A_3}{A_4})\sharp(\ii{A_3}{A_3})\models(\oo{A_3}{A_4})\qquad
(\ee{A_2}{A_3})\sharp(\ii{A_1}{A_2})\models(\oo{A_1}{A_3})\]
thus obtaining
\[(\ee{A_3}{A_4})\sharp(\ii{A_3}{A_3})\sharp(\ee{A_2}{A_3})\sharp(\ii{A_1}{A_2})
\models(\oo{A_3}{A_4})\sharp(\ee{A_2}{A_3})\sharp(\ii{A_1}{A_2})
\models(\oo{A_3}{A_4})\sharp(\oo{A_1}{A_3})\]
giving rise to a term which cannot be further rewritten.
\item[(iii)] the term
\[(\ee{A_4}{A_5})\sharp(\ii{A_3}{A_4})\sharp(\aa{A_3}{A_2})^{\circ}\sharp(\ee{A_1}{A_2})\]
can be rewritten in the following ways
\[(\ee{A_4}{A_5})\sharp(\ii{A_3}{A_4})\sharp(\aa{A_3}{A_2})^{\circ}\sharp(\ee{A_1}{A_2})\models
(\ee{A_4}{A_5})\sharp(\ii{A_3}{A_4})\sharp(\ee{A_1}{A_3})\models(\oo{A_3}{A_5})
\sharp(\ee{A_1}{A_3})\]
\[(\ee{A_4}{A_5})\sharp(\ii{A_3}{A_4})\sharp(\aa{A_3}{A_2})^{\circ}\sharp(\ee{A_1}{A_2})
\models(\oo{A_3}{A_5})\sharp(\aa{A_3}{A_2})^{\circ}\sharp(\ee{A_1}{A_2})
\models(\oo{A_2}{A_5})\sharp(\ee{A_1}{A_2})\]
\[(\ee{A_4}{A_5})\sharp(\ii{A_3}{A_4})\sharp(\aa{A_3}{A_2})^{\circ}\sharp(\ee{A_1}{A_2})
\models(\oo{A_3}{A_5})\sharp(\aa{A_3}{A_2})^{\circ}\sharp(\ee{A_1}{A_2})
\models(\oo{A_3}{A_5})\sharp(\ee{A_1}{A_3})\]
\[(\ee{A_4}{A_5})\sharp(\ii{A_3}{A_4})\sharp(\aa{A_3}{A_2})^{\circ}\sharp(\ee{A_1}{A_2})
\models(\ee{A_4}{A_5})\sharp(\ii{A_2}{A_4})\sharp(\ee{A_1}{A_2})
\models(\oo{A_2}{A_5})\sharp(\ee{A_1}{A_2})\]
by appplying evident rewrite rules to obtain terms which cannot be
further rewritten, coinciding up to renaming of term-variables.
\item[(iv)] the term
\[(\aa{A_5}{A_4})^{\circ}\sharp(\ee{A_3}{A_4})\sharp(\aa{A_3}{A_2})^{\circ}
\sharp(\ee{A_1}{A_2})\sharp(\ii{A_1}{A_1})\]
can be rewritten as
\[\begin{array}{ll}
(\aa{A_5}{A_4})^{\circ}\sharp(\ee{A_3}{A_4})\sharp(\aa{A_3}{A_2})^{\circ}
\sharp(\ee{A_1}{A_2})\sharp(\ii{A_1}{A_1})\models\\
\models
(\aa{A_5}{A_4})^{\circ}\sharp(\ee{A_3}{A_4})\sharp(\aa{A_3}{A_2})^{\circ}\sharp(\oo{A_1}{A_2})
\models\\
(\aa{A_5}{A_4})^{\circ}\sharp(\ee{A_2}{A_4})\sharp(\oo{A_1}{A_2})\models\\
\models(\ee{A_2}{A_5})\sharp(\oo{A_1}{A_2})
\end{array}\]
or as
\[\begin{array}{ll}
(\aa{A_5}{A_4})^{\circ}\sharp(\ee{A_3}{A_4})\sharp(\aa{A_3}{A_2})^{\circ}
\sharp(\ee{A_1}{A_2})\sharp(\ii{A_1}{A_1})
\models\\
\models(\ee{A_3}{A_5})\sharp(\aa{A_3}{A_2})^{\circ}\sharp(\ee{A_1}{A_2})\sharp(\ii{A_1}{A_1})
\models\\
\models(\ee{A_3}{A_5})\sharp(\ee{A_1}{A_3})\sharp(\ii{A_1}{A_1})\models\\
\models(\ee{A_3}{A_5})\sharp(\oo{A_1}{A_3})
\end{array}\]
by applying evident rewrite rules to obtain terms which cannot be
further rewritten, coinciding up to renaming of term-variables.
\end{itemize}
\end{example}
Thinking of the $2$-polygraph for the calculus of $n$-term syllogisms as to a
term rewriting system, the questions on being it terminating
and confluent naturally arise. 
With respect to termination, it must be observed that the length of
the terms that undergo reduction by the application of any of the
non-trivial rewrite rules, 
strictly decreases. The sole exception is represented by the laws of
simple conversion, whose application on the other hand cannot be indefinetely iterated.
In order to prove that a rewriting
system is terminating it suffices to embed it in a rewriting system which is already
known to be such, see~\cite{MR1629216}, typically the set
$\cateb{N}$ of natural numbers together with the ``greater than'' relation
$\gt$. In the case of the $2$-polygraph for the calculus of $n$-term
syllogisms, it suffices to consider the function
$f:\Sigma_1^*\arr\cateb{N}$ which, to each term, assigns its length
and observe that an infinite chain of length-decreasing
reductions would induce an infinite descending chain in $\cateb{N}$.\\

With respect to confluence, things are more delicate. The idea is that
of proving the local confluence of \cate{S} and then conclude by applying
Newman's lemma~\ref{Newman}. Once the number $n$ of occurring term-variables
have been fixed,
proving local confluence of \cate{S}
amounts to testing the effect of the application
of two different rewrite rules on the same subterm of an arbitrary
term. In doing this only the non-trivial rewrite rules have to be taken into account, since
otherwise no significant rewriting takes place. Moreover,
the interesting case is that of the application of such non-trivial
rewrite rules to overlapping subterms. In fact, in general, if
$xs_1ys_2z$ is a term in which the non-overlapping subterms $s_1$ and $s_2$
occurr, then the application of rewrite rules $R_1$ and $R_2$ to them,
yielding the different terms $\sigma_1$ and $\sigma_2$ say,
can always be made confluent as illustrated by the following diagram:
\[\xymatrix{xs_1ys_2z\ar[d]_{R_2}\ar[r]^{R_1}&x\sigma_1ys_2x\ar[d]^{R_2}\\
xs_1y\sigma_2z\ar[r]_{R_1}&x\sigma_1y\sigma_2z}\] 
\begin{theorem}\label{aupt}
For every positive natural number $n$, $n\geq 2$, the $2$-polygraph for the
calculus of $n$-term sylllogisms is locally confluent up to renaming
of term-variables. 
\end{theorem}
\begin{proof}
We proceed by cases:
\begin{itemize}
\item[$n=2$:] The interesting cases are those that amount to the
  application of non-trivial rewrite rules to subterms in which
  exactly two distinct term-variables occurr, because of the condition
  $1\leq i\lt j\leq 2$. Thus, the only case to test is provided by the term
  $(\ee{A_j}{A_i})^{\circ}\sharp(\ii{A_i}{A_i})$, to be considered as
  the overlap of $(\ee{A_j}{A_i})^{\circ}\sharp(\ii{A_i}{A_i})$ itself
  with $(\ee{A_j}{A_i})^{\circ}$. The local confluence
  on this subterm is shown by the reductions
\[(\ee{A_j}{A_i})^{\circ}\sharp(\ii{A_i}{A_i})\models(\oo{A_i}{A_j})\]
\[(\ee{A_j}{A_i})^{\circ}\sharp(\ii{A_i}{A_i})\models(\ee{A_i}{A_j})\sharp(\ii{A_i}{A_i})
\models(\oo{A_i}{A_j})\]
in which the evident rewrite rules have been applied.
\item[$n=3$:] The interesting cases are those that amount to the
  application of non-trivial rewrite rules to subterms in which 
  exactly two or exactly three distinct term-variable occurr, because
  of the conditions $1\leq i\lt j\leq 3$ and $1\leq i\lt j\lt k\leq
  3$. Thus, with respect to the subterms of length two, the cases to test are
\[\begin{tabular}{c c}
$(\ee{A_k}{A_j})^{\circ}\sharp(\aa{A_i}{A_j})$ & 
$(\aa{A_k}{A_j})^{\circ}\sharp(\ee{A_j}{A_i})^{\circ}$\\
$(\ee{A_k}{A_j})^{\circ}\sharp(\ii{A_i}{A_j})$ &
$(\ee{A_j}{A_k})\sharp(\ii{A_j}{A_i})^{\circ}$\\
$(\aa{A_j}{A_k})\sharp(\ii{A_j}{A_i})^{\circ}$ & 
$(\ee{A_j}{A_k})^{\circ}\sharp(\ii{A_j}{A_i})^{\circ}$\\
$(\ii{A_k}{A_j})^{\circ}\sharp(\aa{A_j}{A_i})^{\circ}$ &
$(\ee{A_j}{A_i})^{\circ}\sharp(\ii{A_i}{A_i})$. 
\end{tabular}\]
all of which are easily seen to be locally confluent.
With respect to the subterms of length three, the interesting cases are
\[\begin{tabular}{c c}
$(\aa{A_j}{A_k})\sharp(\aa{A_i}{A_j})\sharp(\ii{A_i}{A_i})$ &
$(\ee{A_j}{A_k})\sharp(\aa{A_i}{A_j})\sharp(\ii{A_i}{A_i})$\\
$(\aa{A_k}{A_j})^{\circ}\sharp(\ee{A_i}{A_j})\sharp(\ii{A_i}{A_i})$ &
$(\ee{A_k}{A_j})^{\circ}\sharp(\aa{A_i}{A_j})\sharp(\ii{A_i}{A_i})$\\
$(\ee{A_j}{A_k})\sharp(\ii{A_j}{A_j})\sharp(\aa{A_j}{A_i})^{\circ}$ &
$(\aa{A_k}{A_j})^{\circ}\sharp(\ee{A_j}{A_i})^{\circ}\sharp(\ii{A_i}{A_i})$\\
$(\ee{A_k}{A_j})^{\circ}\sharp(\ii{A_j}{A_j})\sharp(\aa{A_j}{A_i})^{\circ}$
& $(\aa{A_j}{A_k})\sharp(\ii{A_j}{A_j})\sharp(\aa{A_j}{A_i})^{\circ}$
\end{tabular}\]
all of which are easily seen to be locally confluent.
\item[$n=4$:] The interesting cases are those that amount to the
  application of non-trivial rewrite rules to subterms in which 
  exactly two, three or four distinct term-variables occurr,
  because of the conditions $1\leq i\lt j\leq 4$,
$1\leq i\lt j\lt k\leq 4$ and $1\leq i\lt
  j\lt k\lt l\leq 4$. The interesting 
subterms of length two and three that satisfy the first and the
  second set of conditions, respectively, have
  been previously considered. There remain the subterms of length
  three satisfying the third set of conditions. With respect to these,
the interesting cases are 
\[\begin{tabular}{c c}
$(\aa{A_k}{A_l})\sharp(\aa{A_j}{A_k})\sharp(\aa{A_i}{A_j})$ &
$(\aa{A_k}{A_l})\sharp(\aa{A_j}{A_k})\sharp(\ii{A_i}{A_j})$\\
$(\ee{A_k}{A_l})\sharp(\aa{A_j}{A_k})\sharp(\aa{A_i}{A_j})$ & 
$(\ee{A_l}{A_k})^{\circ}\sharp(\aa{A_j}{A_k})\sharp(\aa{A_i}{A_j})$\\
$(\ee{A_l}{A_k})^{\circ}\sharp(\aa{A_j}{A_k})\sharp(\ii{A_i}{A_j})$ &
$(\aa{A_l}{A_k})^{\circ}\sharp(\ee{A_j}{A_k})\sharp(\aa{A_i}{A_j})$\\
$(\aa{A_l}{A_k})^{\circ}\sharp(\ee{A_j}{A_k})\sharp(\ii{A_i}{A_j})$ &
$(\ii{A_k}{A_l})\sharp(\aa{A_k}{A_j})^{\circ}\sharp(\ee{A_i}{A_j})$\\
$(\ii{A_k}{A_l})\sharp(\aa{A_k}{A_j})^{\circ}\sharp(\oo{A_i}{A_j})$ &
$(\ee{A_l}{A_k})^{\circ}\sharp(\ii{A_j}{A_k})\sharp(\aa{A_j}{A_i})^{\circ}$\\
$(\oo{A_k}{A_l})\sharp(\aa{A_k}{A_j})^{\circ}\sharp(\oo{A_i}{A_j})$ &
$(\oo{A_k}{A_l})\sharp(\aa{A_k}{A_j})^{\circ}\sharp(\ee{A_i}{A_j})$\\
$(\aa{A_l}{A_k})^{\circ}\sharp(\ee{A_k}{A_j})^{\circ}\sharp(\ii{A_j}{A_i})^{\circ}$&
$(\ii{A_l}{A_k})^{\circ}\sharp(\aa{A_k}{A_j})^{\circ}\sharp(\ee{A_j}{A_i})^{\circ}$\\
$(\aa{A_l}{A_k})^{\circ}\sharp(\ee{A_k}{A_j})^{\circ}\sharp(\aa{A_i}{A_j})$&
$(\ii{A_l}{A_k})^{\circ}\sharp(\aa{A_k}{A_j})^{\circ}\sharp(\ee{A_i}{A_j})$\\
$(\ii{A_l}{A_k})^{\circ}\sharp(\aa{A_k}{A_j})^{\circ}\sharp(\oo{A_i}{A_j})$&
$(\ee{A_l}{A_k})^{\circ}\sharp(\ii{A_k}{A_j})^{\circ}\sharp(\aa{A_j}{A_i})^{\circ}$\\
$(\aa{A_l}{A_k})^{\circ}\sharp(\ee{A_j}{A_k})\sharp(\ii{A_j}{A_i})^{\circ}$&
$(\aa{A_l}{A_k})^{\circ}\sharp(\ee{A_k}{A_j})^{\circ}\sharp(\aa{A_i}{A_j})$\\
$(\aa{A_l}{A_k})^{\circ}\sharp(\ee{A_k}{A_j})^{\circ}\sharp(\ii{A_i}{A_j})$&
$(\ii{A_l}{A_k})\sharp(\aa{A_k}{A_j})^{\circ}\sharp(\ee{A_j}{A_i})^{\circ}$\\
$(\oo{A_k}{A_j})\sharp(\aa{A_k}{A_j})^{\circ}\sharp(\ee{A_j}{A_i})$&$(\oo{A_k}{A_l})
\sharp(\aa{A_k}{A_j})^{\circ}\sharp(\ee{A_j}{A_i})^{\circ}$

\end{tabular}\]
all of which can be easily seen to be locally confluent up to renaming
of term-variables, as for example in the case of
$(\oo{A_k}{A_l})\sharp(\aa{A_k}{A_j})^{\circ}\sharp(\ee{A_j}{A_i})^{\circ}$,
which on one hand reduces to $(\oo{A_j}{A_l})\sharp(\ee{A_i}{A_j})$
and on the other hand reduces to
$(\oo{A_k}{A_l})\sharp(\ee{A_i}{A_k})$.\\
The interesting subterms of length four that satisfy the condition $1\leq i\lt j\lt
k\lt l\leq 4$ must contain a premise of the form $(\ii{A_i}{A_i})$,
$(\ii{A_j}{A_j})$, $(\ii{A_k}{A_k})$ or $(\ii{A_l}{A_l})$. They fall
inside the already tested cases, by thinking of them as obtained from
the overlapping of three terms of length two.
\item[$n\geq 5$:] The interesting cases fall inside the already
  tested cases, by thinking of them as obtained from the overlapping of
  a suitable amount of terms of length two.
\end{itemize}
\end{proof}
\begin{corollary}
For every positive natural number $n$, $n\geq 2$, the $2$-polygraph for the
calculus of $n$-term sylllogisms is confluent up to renaming
of term-variables. 
\end{corollary}
\begin{proof}
It follows from theorem~\ref{aupt} and Newman's lemma~\ref{Newman}.
\end{proof}
%further categorical aspects
\bibliographystyle{plain}
\bibliography{BiblioTeX}
\end{document}